\documentclass[12pt]{amsart}
\usepackage{amsmath}
\usepackage{amsthm}
\usepackage{amssymb}
\usepackage{fullpage}
\usepackage{enumitem}
\usepackage{pgfplots}
\pgfplotsset{compat=1.18}
\usepackage[hidelinks]{hyperref}

\newtheorem{theorem}{Theorem}[section]
\newtheorem*{theorem*}{Theorem}
\newtheorem{lemma}[theorem]{Lemma}

\newtheorem{proposition}[theorem]{Proposition}

\theoremstyle{definition}
\newtheorem{example}[theorem]{Example}
\newtheorem{definition}[theorem]{Definition}
\newtheorem{remark}[theorem]{Remark}

\newcommand{\seg}[4]{
  \def\argarrow{#2}
  \def\rightarrowtoken{>}
  \def\leftarrowtoken{<}
  \ifx\argarrow\rightarrowtoken
    \pi_{#1 \rightarrow #3}
  \else
    \ifx\argarrow\leftarrowtoken
      \pi_{#1 \leftarrow #3}
    \else
      \errmessage{Invalid arrow direction '#2'. Use '<' or '>'.}
    \fi
  \fi
}

\newcommand{\sseg}[5]{
  \def\argarrow{#2}
  \def\rightarrowtoken{>}
  \def\leftarrowtoken{<}
  \ifx\argarrow\rightarrowtoken
    {#5}_{#1 \rightarrow #3}
  \else
    \ifx\argarrow\leftarrowtoken
      {#5}_{#1 \leftarrow #3}
    \else
      \errmessage{Invalid arrow direction '#2'. Use '<' or '>'.}
    \fi
  \fi
}

\DeclareMathOperator{\rev}{rev}

\title{The Dynamics and Orbit Structure of the Topdrop Map}

\author{Nathan R. Krause}

\begin{document}

\begin{abstract}
We study the topdrop map, a mapping on permutations in \(S_n\) related to card shuffling. We show this map is bijective and study its orbit structure. We introduce the notion of the topdrop-necklace as a way of classifying the orbits of the map and establish a general theorem to count orbits using topdrop-necklaces. We then provide exact counts for orbits of size two through five and lower bounds for the number of orbits of sizes six and eight. We show symmetries in orbits which happen when \(n\) or \(n-1\) is in the topdrop-necklace, count these orbits, and show that they have even size.  We prove a restriction on topdrop-necklaces based on permutation parity.
\end{abstract}

\maketitle

 \section{Introduction}
In 1976, David Berman and Murray S. Klamkin \cite{BermanKlamkinKnuth1976} described a ``reverse card shuffle" process in which one looks at the top card in a deck of \(n\) cards (numbered \(1\) thru \(n\)), takes that many cards off the top, reverses them, and puts them back on the top. According to Donald Knuth in an editorial note \cite{BermanKlamkinKnuth1977}, this operation was in fact first proposed by John Conway with the name \textit{topswaps}  (more commonly spelled  \textit{topswops}). In the language of permutations, the topswop map is applied to a permutation by taking the segment from the front of the permutation which is as large as the first element, reversing the segment, and placing it at the {front}. For instance, the topswop map applied to the permutation $435126$ results in the permutation $153426$.
Martin Gardner included topswops in his book ``Time Travel and other Mathematical Bewilderments" published in 1988 \cite{Gardner}. 
He also describes several variants, which are also attributed to Conway. One of these variants is \emph{topdrops}, which is the main focus of this paper. 

In topdrops, instead of placing the reversed segment at the front of the permutation, one places it at the end. We denote the topdrop map by \(T\). For example, the topdrop map applied to $435126$ results in $T(435126)=261534$;  the first element is $4$, so the first four elements are reversed and placed at the end. This results in behavior that is very different from topswops. In particular, while 
repeated application of the topswop map always terminates with a 1 on top of the deck (proved by Herbert Wilf, as recounted in \cite{Gardner}), we show that the topdrop map is, in fact, bijective. Thus, a main aim of this paper is to study the \emph{orbits} of the topdrop map. 

We begin in Section \ref{sec:Fundamental Properties of the Topdrop Map}, where we provide useful notation for working with the topdrop map, prove that the map is bijective by showing its inverse, and define the orbit of a permutation. We also prove a useful lemma that applying the inverse of the topdrop map \(k\) times is the same as reversing the permutation, applying the topdrop map \(k\) times, and reversing the result. That is, 
\newtheorem*{lem:revinv}{Lemma \ref{lem:revinv}}
\begin{lem:revinv}
For any \(\pi\) in \(S_n\) and any integer \(k\), \[T^{-1}(\pi) = \rev(T(\rev(\pi))).\]
\end{lem:revinv}

In Section \ref{sec:Topdrop-necklaces}, we introduce the topdrop-necklace, which is the sequence of first elements of permutations in an orbit. We show special properties and symmetries that occur when \(n\) or \(n-1\) appears in an orbit. These are useful for proving the later theorems about orbit counts.

\newtheorem*{thm:necklaces_with_n_or_n_minus_one_special_properties}{Theorem \ref{thm:necklaces_with_n_or_n_minus_one_special_properties}}
\begin{thm:necklaces_with_n_or_n_minus_one_special_properties}
The following is true when \(\pi\) is in \(S_n\) and \(\pi_1 = n\) or \(\pi_1 = n-1\), for \(n \geq 2\):
\begin{enumerate}
\item \(T^{-k}(\pi)=\rev(T^{k+1}(\pi))\) for any integer \(k\).

\item The orbit of \(\pi\) has even size. Equivalently, if a topdrop-valid necklace of \(S_n\) contains \(n\) and/or \(n-1\), then it has even size.

\item Let \(2m\) be the orbit size of \(\pi\). Then \(T^{m}(\pi)_1 = n\) or \(T^{m}(\pi)_1 = n-1\). Hence \(n\) and \(n-1\) collectively appear at least twice in the topdrop-necklace. In fact, they appear exactly twice.

\item There are exactly \((n-1)!\) orbits in \(S_n\) which have \(n\) and/or \(n-1\) in their topdrop-necklaces.
\end{enumerate}
\end{thm:necklaces_with_n_or_n_minus_one_special_properties}

\newtheorem*{thm:necklaces_with_n_or_n_minus_one_additional_symmetries}{Theorem \ref{thm:necklaces_with_n_or_n_minus_one_additional_symmetries}}
\begin{thm:necklaces_with_n_or_n_minus_one_additional_symmetries}
Assume \(\pi\) is in \(S_n\) and \(\pi_1 = n\) or \(\pi_1 = n-1\), for \(n \geq 2\). Relabel \(T^{k}(\pi)\) as \(\pi'\), \(T^{-k}(\pi)\) as \(\pi''\), and \(T^{-k+2}(\pi)\) as \(\pi'''\). Then the following is true:
\begin{enumerate}
\item \(\sseg{1}{>}{\pi'_1}{\pi'_1}{\pi'} = \sseg{1}{>}{\pi'_1}{\pi'_1}{\pi''}\) and \(\sseg{\pi'_1+1}{>}{n}{n-\pi'_1}{\pi'} = \sseg{\pi'_1+1}{<}{n}{n-\pi'_1}{\pi''}\). In other words, \(T^k(\pi)\) and \(T^{-k}(\pi)\) have the first \(T^{k}(\pi)_1\) elements in common, while the last \(n-T^{k}(\pi)_1\) elements are reversed.

\item \(\sseg{1}{>}{n-\pi'_n}{n-\pi'_n}{\pi'} = \sseg{1}{<}{n-\pi'_n}{n-\pi'_n}{\pi'''}\) and \(\sseg{n-\pi'_n+1}{>}{n}{\pi'_n}{\pi'} = \sseg{n-\pi'_n+1}{>}{n}{\pi'_n}{\pi'''}\). In other words, \(T^k(\pi)\) and \(T^{-k+2}(\pi)\) have the last \(T^{k+2}(\pi)_n\) elements in common, while the first \(n-T^{k+2}(\pi)_n\) elements are reversed.
\end{enumerate}
\end{thm:necklaces_with_n_or_n_minus_one_additional_symmetries}

In Section \ref{sec:Orbit Counting}, we show that orbits can be counted via topdrop-necklaces. First, we define the set of all valid topdrop-necklaces \(A_n\) and introduce the fundamental period \(P(N)\) of a topdrop-necklace. Then we state two theorems.

\newtheorem*{thm:orbits_with_necklace}{Theorem \ref{thm:orbits_with_necklace}}
\begin{thm:orbits_with_necklace}
Given a necklace \(\mathcal{N}\) in \(A_n\), the number of distinct orbits with necklace \(\mathcal{N}\) is \[\frac{\left(n-\left(\textnormal{\# of distinct values in }\mathcal{N}\right)\right)!}{P(\mathcal{N})}\]
\end{thm:orbits_with_necklace}

\newtheorem*{thm:orbit_counting}{Theorem \ref{thm:orbit_counting}}
\begin{thm:orbit_counting}
For every natural number \(n\) and natural number \(k\), the number of orbits of size \(k\) in \(S_n\) is \[\sum_{\{\mathcal{N} \in A_n \,:\, |\mathcal{N}|=k\}}{\frac{\left(n-\left(\textnormal{\# of distinct values in }\mathcal{N}\right)\right)!}{P(\mathcal{N})}}\]
\end{thm:orbit_counting}

We use these results to provide a theorem about orbit counts for small orbit sizes.

\newtheorem*{thm:orbit_counts}{Theorem \ref{thm:orbit_counts}}
\begin{thm:orbit_counts}
The following is true of the topdrop map in \(S_n\):
\begin{enumerate}
    \item There are \((n-2)!\) orbits of size two, when \(n \geq 2\)
    \item There are no orbits of size three
    \item There are \((n-2)!\) orbits of size four, when \(n \geq 3\)
    \item The number of orbits of size five, when \(n \geq 7\), is 
        \begin{itemize}
        \item \((n-6)(n-5)!\) orbits of size five if \(n\) is even
        \item \((n-5)(n-5)!\) orbits of size five if \(n\) is odd
        \end{itemize}
    \item The number of orbits of size six, when \(n \geq 4\), is at least        \begin{itemize}
        \item \((2n^2-11n+14)(n-4)!\) orbits of size six when \(n\) is even
        \item \((2n^2-12n+18)(n-4)!\) orbits of size six when \(n\) is odd
        \end{itemize}
    \item The number of orbits of size eight, when \(n \geq 5\), is at least        \begin{itemize}
        \item \(\left( \frac{3}{2}n^3 - \frac{35}{2}n^2 + 64n - 70 \right) (n - 5)!\) if \(n\) is congruent to \(0\) or \(2\) modulo \(6\)
        \item \(\left( \frac{3}{2}n^3 - \frac{35}{2}n^2+ \frac{137}{2}n - \frac{181}{2} \right) (n - 5)!\) if \(n\) is congruent to \(1\) modulo \(6\)
        \item \(\left( \frac{3}{2}n^3 - \frac{35}{2}n^2 + \frac{135}{2}n - \frac{171}{2} \right) (n-5)!\) if \(n\) is congruent to \(3\) or \(5\) modulo \(6\)
        \item \(\left( \frac{3}{2}n^3 - \frac{35}{2}n^2 + 66n - 80 \right) (n-5)!\) if \(n\) is congruent to \(4\) modulo \(6\)
        \end{itemize}
\end{enumerate}
\end{thm:orbit_counts}

In Section \ref{sec:Permutation Parity and the Topdrop-necklace}, we use permutation parity to make a statement about valid topdrop-necklaces.
\newtheorem*{thm:parity}{Theorem \ref{thm:parity}}
\begin{thm:parity}
Let \([x_1, x_2, \ldots, x_k]\) be a topdrop-valid necklace in \(S_n\). If \(n\) is even, then evenly many \(x_i\)'s are congruent to \(1\) or \(2\) modulo \(4\). If \(n\) is odd, then evenly many \(x_i\)'s are congruent to \(2\) or \(3\) modulo \(4\).
\end{thm:parity}

We end in Section \ref{sec:Future Directions} with some ideas for future work. In the appendix, we give more details on the history of the topswops map.

\section{Fundamental Properties of the Topdrop Map}\label{sec:Fundamental Properties of the Topdrop Map}

In this section, we define some notation, formally define the topdrop map, and prove Lemma \ref{lem:revinv} relating the inverse of the topdrop map and the reverse of the permutation.

We refer to elements of permutations by their indices using subscripts, i.e. if \(\pi\) is in \(S_n\), then \(\pi_i\) is the element in index \(i\) of \(\pi\). We write out permutations by listing their elements left to right, e.g. if \(\pi\) is in \(S_3\), then \(\pi = \pi_1 \pi_2 \pi_3\).

As shorthand, we write segments of permutations of the form \(\pi_a \pi_{a+1} \ldots \pi_{b}\) by \(\seg{a}{>}{b}{b+1-a}\). The subscript \(a \rightarrow b\) indicates that \(a\) is the lesser index and \(b\) is the greater index. We denote \(\pi_b \pi_{b-1} \ldots \pi_{a}\) by \(\seg{b}{<}{a}{b+1-a}\). The subscript \(b \leftarrow a\) indicates that \(a\) is the lesser index and \(b\) is the greater index.

We adopt the convention that if \(a > b\), then \(\seg{a}{>}{b}{0}\) is empty. Similarly \(\seg{b}{<}{a}{0}\) is empty. For example, \(\seg{1}{>}{3}{3}\seg{4}{>}{3}{0}\seg{3}{<}{4}{0}\seg{7}{<}{4}{4} = \pi_1 \pi_2 \pi_3 \pi_7 \pi_6 \pi_5 \pi_4\).

We are now ready to formally describe the topdrop map.

\begin{definition}
The topdrop map \(T: S_n \to S_n\) is defined by \[T\left(\seg{1}{>}{n}{n}\right) = \seg{\pi_1+1}{>}{n}{n-\pi_1}\seg{\pi_1}{<}{1}{\pi_1}\] 
\end{definition}

Note that if \(\pi_1 = n\) or \(\pi_1 = n-1\), then the topdrop map reverses the permutation.

\begin{example}
\(T(4612375) = 3752164\).
\(T(321) = 123\).
\(T(41253) = 35214\).
\end{example}

After applying the topdrop map, it is clear how to reverse the operation by looking at the bottom card, which leads to the next proposition.

\begin{proposition}
\(T\) is bijective.
\end{proposition}

\begin{proof}
We define \(U: S_n \to S_n\) and show that it is the inverse of \(T\). Let
\[U(\seg{1}{>}{n}{n}) = \seg{n}{<}{n-\pi_n+1}{\pi_n}\seg{1}{>}{n-\pi_n}{n-\pi_n}\]

Then \[U(T(\seg{1}{>}{n}{n})) = U(\seg{\pi_1+1}{>}{n}{n-\pi_1}\seg{\pi_1}{<}{1}{\pi_1}) = \seg{1}{>}{\pi_1}{\pi_1}\seg{\pi_1+1}{>}{n}{n-\pi_1} = \seg{1}{>}{n}{n}\] 
\[T(U(\seg{1}{>}{n}{n})) = T(\seg{n}{<}{n-\pi_n+1}{\pi_n}\seg{1}{>}{n-\pi_n}{n-\pi_n}) = \seg{1}{>}{n-\pi_n}{n-\pi_n}\seg{n-\pi_n+1}{>}{n}{\pi_n} = \seg{1}{>}{n}{n}\]
\end{proof}

From now on we will refer to \(U\) as \(T^{-1}\). We denote by \(T^k\) the \(k\)-fold composition of \(T\), and by \(T^{-k}\) the \(k\)-fold composition of \(T^{-1}\). \(T^{0}\) is the identity map.

\begin{proposition}
For any \(\pi\) in \(S_n\) there is a natural number \(s\) so that \(T^{s}\left(\pi\right) = \pi\).
\end{proposition}

\begin{proof}
This follows from the finiteness of \(S_n\) and that \(T\) is bijective.
\end{proof}

Since we have guaranteed that we eventually arrive back to the original permutation, we are ready to define the orbit of a permutation.

\begin{definition}
The orbit \(\mathcal{O}{(\pi)}\) of a permutation is the sequence \[
(\pi, T^{1}(\pi), T^{2}(\pi), \ldots, T^{s-1}(\pi)),
\] where \(s\) is the smallest natural number so that \(T^{s}\left(\pi\right) = \pi\). We refer to \(s\) as the size of the orbit.
\end{definition}

We denote the reverse of a permutation \(\pi\) with \(\rev(\pi)\). For example, \(\rev(1324) = 4231\). The following lemma gives a useful symmetry which we will use multiple times. It says that applying the topdrop map a number of times is the same as reversing the permutation, applying the inverse the same number of times, and reversing again.

\begin{lemma}\label{lem:revinv}
For any \(\pi\) in \(S_n\) and any integer \(k\), \[T^{-1}(\pi) = \rev(T(\rev(\pi))).\]
\end{lemma}
\begin{proof}
By definition, \[T^{-1}(\pi) = \seg{n}{<}{n-\pi_n+1}{\pi_n} \seg{1}{>}{n-\pi_n}{n-\pi_n}.\] On the other hand, \[T(\rev(\pi)) = T(\seg{1}{>}{n}{n}) = \seg{n-\pi_n}{<}{1}{n-\pi_n} \seg{n-\pi_n+1}{>}{n}{\pi_n},\] so \[\rev(T(\rev(\pi))) = \seg{n}{<}{n-\pi_n+1}{\pi_n} \seg{1}{>}{n-\pi_n}{n-\pi_n} = T^{-1}(\pi).\]
\end{proof}

\begin{example}
Let \(\pi = 231\). \(T^{-1}(\pi) = 123\). \(\rev(\pi) = 132\), \(T(\rev(\pi)) = 321\), and \(\rev(T(\rev(\pi))) = 123 = T^{-1}(\pi)\).
\end{example}

\section{Topdrop-necklaces}\label{sec:Topdrop-necklaces}
It is very helpful when studying the dynamics of the topdrop map to track the sequence of first elements of permutations in an orbit. We encapsulate this data by defining the topdrop-necklace. We will also prove some special properties that occur when \(n\) or \(n-1\) is in the topdrop-necklace, which will be useful for the proofs in the next section.

\begin{definition}
The topdrop-necklace \(\mathcal{N}({\pi})\) of a permutation is the necklace formed by taking the first element of every permutation of \(\mathcal{O}{(\pi)}\) in order. The size of \(\mathcal{N}({\pi})\) is the same size as \(\mathcal{O}{(\pi)}\). We denote the sequence with brackets instead of parentheses to emphasize that the topdrop-necklace is not a permutation. When we wish to refer to a topdrop-necklace without referring to a particular permutation which generates it, we may just write \(\mathcal{N}\). Often, we just say necklace instead of topdrop-necklace, and these should be understood to both refer to the topdrop-necklace.
\end{definition}

\begin{example}
Let \(\pi = 14235\) in \(S_5\). Then \(\mathcal{O}(\pi)\) is \((14235,42351,\allowbreak15324,\allowbreak 53241)\). The size of \(\mathcal{O}(\pi)\) is \(4\). The topdrop-necklace \(\mathcal{N}(\pi)\) is \([1,4,1,5]\). The size of \(\mathcal{N}(\pi)\) is also \(4\).
\end{example}

\begin{definition}
A necklace is topdrop-valid in \(S_n\) if there is some permutation \(\pi\) in \(S_n\) having that sequence as its topdrop-necklace. We let \(A_n\) be the set of all topdrop-valid necklaces in \(S_n\).
\end{definition}

\begin{example}
The sequence \([1,4,1,5]\) is a topdrop-valid necklace in \(S_5\) per the previous example. The sequence is \([5,5]\) is not a topdrop-valid necklace in \(S_5\) because there is no permutation in \(S_5\) which can have \([5,5]\) as its topdrop-necklace.
\end{example}

We have to be careful when looking at topdrop-valid necklaces, because sometimes as we iterate the topdrop map the first elements start to repeat before we have reached the original permutation. However, it is still useful to track how long it takes for the first elements to repeat, as we will soon see.

\begin{definition}
For a necklace \(\mathcal{N}\), there is some smallest contiguous segment of the necklace which can be repeated  to form the necklace. The \textit{fundamental period} \(P(\mathcal{N}))\) of \(\mathcal{N}\) is the number of copies of that segment required to form the necklace.
\end{definition}

\begin{example}
The necklace \([3, 5, 9, 8, 6, 3, 5, 9, 8, 6, 3, 5, 9, 8, 6, 3, 5, 9, 8, 6]\) is topdrop-valid in \(S_{12}\). Its fundamental period is \(4\), because it is formed from \(4\) repetitions of \(3, 5, 9, 8, 6\).
\end{example}

When \(n\) or \(n-1\) appears in a topdrop-necklace, there are special properties and symmetries which arise.

\begin{theorem}\label{thm:necklaces_with_n_or_n_minus_one_special_properties}
The following is true when \(\pi\) is in \(S_n\) and \(\pi_1 = n\) or \(\pi_1 = n-1\), for \(n \geq 2\):
\begin{enumerate}
\item \(T^{-k}(\pi)=\rev(T^{k+1}(\pi))\) for any integer \(k\).

\item The orbit of \(\pi\) has even size. Equivalently, if a topdrop-valid necklace of \(S_n\) contains \(n\) and/or \(n-1\), then it has even size.

\item Let \(2m\) be the orbit size of \(\pi\). Then \(T^{m}(\pi)_1 = n\) or \(T^{m}(\pi)_1 = n-1\). Hence \(n\) and \(n-1\) collectively appear at least twice in the topdrop-necklace. In fact, they appear exactly twice.

\item There are exactly \((n-1)!\) orbits in \(S_n\) which have \(n\) and/or \(n-1\) in their topdrop-necklaces.
\end{enumerate}
\end{theorem}

\begin{proof}[Proof of part (1)]
By Lemma \ref{lem:revinv}, \(T^{-k}(\pi) = \rev(T^{k}(\rev(\pi)))\). Since \(\pi_1 = n\) or \(\pi_1 = n-1\), we have \(\rev(\pi) = T(\pi)\).  So \(T^{-k} = \rev(T^{k}(T(\pi))) = \rev(T^{k+1}(\pi))\).
\end{proof}

\begin{proof}[Proof of part (2)]
Assume for contradiction that there is a \(\pi\) satisfying \(\pi_1 = n\) or \(\pi_1 = n-1\) with odd orbit size, so \(\mathcal{O}(\pi) = 2m+1\) for some natural number \(m\). By part (1), \(T^{-m}(\pi) = \rev(T^{m+1}(\pi))\). Since \(T^{-m}(\pi)\) is a permutation in \(\mathcal{O}(\pi)\), \(T^{-m}(\pi) = T^{2m+1}(T^{-m}(\pi)) = T^{m+1}(\pi) = \rev(T^{m+1}(\pi))\), which is impossible.
\end{proof}

\begin{proof}[Proof of part (3)]
By Lemma \ref{lem:revinv} we know \[T^{-m}(\pi) = \rev(T^{m}(\rev(\pi)))\]
\[\rev(T^{-m}(\pi)) = T^{m}(\rev(\pi))\]Since \(\pi_1 = n\) or \(\pi_1 = n-1\), \(T^{m}(\rev(\pi)) = T^{m+1}(\pi)\). Since \(T^{-m}(\pi)\) is in \(\mathcal{O}(\pi)\), \(T^{-m}(\pi) = T^{2m}(T^{-m}(\pi)) = T^{m}(\pi)\). So \(T^{m}(\pi) = \rev(T^{m+1}(\pi)),\) which can only happen if \(T^{m}(\pi)_1 = n\) or \(T^{m}(\pi)_1 = n-1\). So \(n\) and \(n-1\) collectively appear at least twice in the topdrop-necklace.

Assume for contradiction that \(n\) and \(n-1\) appear more than twice in the topdrop-necklace of \(\pi\). Then there is some other \(\ell\) so that \(T^{\ell}(\pi)_1 = n\) or \(T^{\ell}(\pi)_1 = n-1\). Assume without loss of generality that \(0 < \ell < m\). So \(T^{\ell+1}(\pi) = \rev(T^{\ell(\pi)})\). By proposition, \(T^{\ell+1}(\pi) = \rev(T^{-\ell}(\pi))\). So \[\rev(T^{\ell(\pi)}) = \rev(T^{-\ell}(\pi))\] \[T^{\ell}(\pi) = T^{-\ell}(\pi).\] But since the orbit size is \(2m\) and \(0 < \ell < m\), \(T^{\ell}(\pi)\) and \(T^{-\ell}(\pi)\) must be distinct permutations in the orbit, so they cannot be equal.
\end{proof}

\begin{proof}[Proof of part (4)]
When \(n = 1\) the statement holds trivially. For \(n \geq 2\), there are \(2(n-1)!\) permutations in \(S_n\) which begin with \(n\) or \(n-1\). Every orbit which has \(n\) or \(n-1\) in its topdrop-necklace contains exactly two of these permutations. So there are \(2(n-1)!/2 = (n-1)!\) such orbits. 
\end{proof}

\begin{example}\label{ex:n_n_minus_one}
Let \(\pi\) = \(6132574\) in \(S_7\).  The size of \(\mathcal{O}(\pi)\) is eight, an even number. For purposes of illustration, we write out the orbit starting from \(T^{-3}(\pi)\).
\begin{align*}
T^{-3}(\pi) &= 2531647\\
T^{-2}(\pi) &= 3164752\\
T^{-1}(\pi) &= 4752613\\
T^{0}(\pi) &= 6132574\\
T^{1}(\pi) &= 4752316\\
T^{2}(\pi) &= 3162574\\
T^{3}(\pi) &= 2574613\\
T^{4}(\pi) &= 7461352
\end{align*} We read from the front elements that the topdrop-necklace is \[[2,3,4,6,4,3,2,7].\] \(n=7\) and \(n-1=6\) collectively appear exactly twice in the necklace. \(6\) appears at the front of \(T^{0}(\pi)\), and \(7\) appears at the front of \(T^{8/2}(\pi) = T^{4}(\pi)\), since the size of the orbit is eight.

If we pick an integer, say \(k=2\), then we see that \[T^{-2}(\pi) = 3164752 = \rev(T^{3}(\pi)).\]
\end{example}

\begin{theorem}\label{thm:necklaces_with_n_or_n_minus_one_additional_symmetries}
Assume \(\pi\) is in \(S_n\) and \(\pi_1 = n\) or \(\pi_1 = n-1\), for \(n \geq 2\). Relabel \(T^{k}(\pi)\) as \(\pi'\), \(T^{-k}(\pi)\) as \(\pi''\), and \(T^{-k+2}(\pi)\) as \(\pi'''\). Then the following is true:
\begin{enumerate}
\item \(\sseg{1}{>}{\pi'_1}{\pi'_1}{\pi'} = \sseg{1}{>}{\pi'_1}{\pi'_1}{\pi''}\) and \(\sseg{\pi'_1+1}{>}{n}{n-\pi'_1}{\pi'} = \sseg{\pi'_1+1}{<}{n}{n-\pi'_1}{\pi''}\). In other words, \(T^k(\pi)\) and \(T^{-k}(\pi)\) have the first \(T^{k}(\pi)_1\) elements in common, while the last \(n-T^{k}(\pi)_1\) elements are reversed.

\item \(\sseg{1}{>}{n-\pi'_n}{n-\pi'_n}{\pi'} = \sseg{1}{<}{n-\pi'_n}{n-\pi'_n}{\pi'''}\) and \(\sseg{n-\pi'_n+1}{>}{n}{\pi'_n}{\pi'} = \sseg{n-\pi'_n+1}{>}{n}{\pi'_n}{\pi'''}\). In other words, \(T^k(\pi)\) and \(T^{-k+2}(\pi)\) have the last \(T^{k+2}(\pi)_n\) elements in common, while the first \(n-T^{k+2}(\pi)_n\) elements are reversed.
\end{enumerate}
\end{theorem}

\begin{proof}
We prove by induction. If \(k = 0\), then 
\[\sseg{1}{>}{n}{n}{\pi'} = T^{0}(\pi) = T^{-0}(\pi) = \sseg{1}{>}{n}{n}{\pi''},\]
which implies that \(\sseg{1}{>}{\pi'_1}{\pi'_1}{\pi'} = \sseg{1}{>}{\pi'_1}{\pi'_1}{\pi''}\). If \(\pi'_1 = n\), then \(\sseg{\pi'_1+1}{>}{n}{n-\pi'_1}{\pi'}\) and \(\sseg{\pi'_1+1}{<}{n}{n-\pi'_1}{\pi''}\) are both empty, so they are equal. If \(\pi'_1 = n-1\), then \[\sseg{\pi'_1+1}{>}{n}{n-\pi'_1}{\pi'} = \pi'_n = \pi''_n =  \sseg{\pi'_1+1}{<}{n}{n-\pi'_1}{\pi''}.\]

Assume that the statement holds for \(k = m\). Label \(T^{-m}(\pi)\) as \(\sigma\). By Lemma \ref{lem:revinv}, \[T^{m+1}(\pi)=\sseg{n}{<}{1}{n}{\sigma} = \sseg{n}{<}{n-\sigma_n+1}{\sigma_n}{\sigma} \sseg{n-\sigma_n}{<}{1}{n-\sigma_n}{\sigma}.\] Meanwhile, \[T^{-(m+1)}(\pi) = T^{-1}(\sigma) = \sseg{n}{<}{n-\sigma_n+1}{\sigma_n}{\sigma} \sseg{n-\sigma_n}{>}{1}{n-\sigma_n}{\sigma}\] and the statement holds true since the first \(\sigma_n\) elements are in common while the last \(n-\sigma_n\) elements are reversed, closing the induction.
\end{proof}

\begin{example}
In Example \ref{ex:n_n_minus_one}, if we pick \(k=2\) then we see that the first \(T^{-2}(\pi)_1 = 3\) elements of \(T^{-2}(\pi)\) are \(316\), which are also the first three elements of \(T^{2}(\pi)\). The last four elements of \(T^{-2}(\pi)\) are \(4752\), while the last four elements of \(T^{2}(\pi)\) are \(2574\), which is \(4752\) reversed.
\end{example}

\begin{proof}[Proof of part (2)]
By Theorem \ref{thm:necklaces_with_n_or_n_minus_one_special_properties}, \(T^{k}(\pi) = \rev(T^{-k+1}(\pi))\) and \(T^{-k+2}(\pi) =\rev(T^{k-1}(\pi)\). By part (1), \(T^{-k+1}(\pi)\) and \(T^{k-1}(\pi)\) have the first \(T^{-k+1}(\pi)_1\) elements in common while the last \(n - T^{-k+1}(\pi)_1\) elements are reversed. Hence \(T^{k}(\pi)\) and \(T^{-k+2}(\pi)\) have the last \(T^{-k+1}(\pi)_1 = T^{k}(\pi)_n\) elements in common while the first \(n-T^{k}(\pi)_n\) elements are reversed, which is to say that  \(\sseg{1}{>}{n-\pi'_n}{n-\pi'_n}{\pi'} = \sseg{1}{<}{n-\pi'_n}{n-\pi'_n}{\pi'''}\) and \(\sseg{n-\pi'_n+1}{>}{n}{\pi'_n}{\pi'} = \sseg{n-\pi'_n+1}{>}{n}{\pi'_n}{\pi'''}\).
\end{proof}

\begin{example}
In Example \ref{ex:n_n_minus_one}, if we pick \(k=2\) then we see that the last \(T^{-2}(\pi)_n = 2\) elements of \(T^{-2}(\pi)\) are \(52\), which are also the last two elements of \(T^{4}(\pi)\). The first five elements of \(T^{-2}(\pi)\) are \(31647\), while the first five elements of \(T^{4}(\pi)\) are \(74613\), which is \(31647\) reversed.
\end{example}

\section{Orbit Counting}\label{sec:Orbit Counting}

The following theorems show that if we know all the topdrop-valid necklaces of size \(k\) in \(S_n\), then we can count the orbits of size \(k\) in \(S_n\).

\begin{theorem}\label{thm:orbits_with_necklace}
Given a necklace \(\mathcal{N}\) in \(A_n\), the number of distinct orbits with necklace \(\mathcal{N}\) is \[\frac{\left(n-\left(\textnormal{\# of distinct values in }\mathcal{N}\right)\right)!}{P(\mathcal{N})}\]
\end{theorem}

\begin{proof}
Fix a representation \([n_1, n_2, \ldots, n_k]\) of \(\mathcal{N}\) (necklaces are equivalent under circular shifts, but we pick a starting \(n_1\)). Every orbit having \(\mathcal{N}\) as its necklace must have contain some permutation \(\pi\) so that \(\pi_1 = n_1\), \(T^{1}(\pi)_1 = n_2\), \(\ldots\), \(T^{k}(\pi)_1 = n_k\). These criteria precisely dictate the positions where the values of \(\mathcal{N}\) must appear in \(\pi\). Any arrangement of the values which do not appear in \(\mathcal{N}\) in the remaining positions in \(\pi\) results in \(\mathcal{O}\) having necklace \(\mathcal{N}\). There are \(\left(n-\textnormal{\# of distinct values in }\mathcal{N}\right)!\) such permutations in \(S_n\). Let \(B\) the set of these permutations.

If the size of \(\mathcal{N}\) is \(s\), let \(j = s/P(\mathcal{N})\). If \(\pi\) is in \(B\) then \(T^{ij+1}(\pi)\) is also in \(B\) for \(0 \leq i \leq P(\mathcal{N}) - 1\) by the definition of the fundamental period, each \(T^{ij}(\pi)\) must be distinct; no other permutations in \(B\) can be in \(\mathcal{O}(\pi)\). So each orbit having necklace \(\mathcal{N}\) contains \(P(\mathcal{N})\) distinct permutations in \(B\), and no permutation can belong to more than one orbit. So because there are \(\left(n-\textnormal{\# of distinct values in }\mathcal{N}\right)!\) permutations in \(B\), there are \(\frac{\left(n-\left(\textnormal{\# of distinct values in }\mathcal{N}\right)\right)!}{P(N)}\) orbits with necklace \(\mathcal{N}\) in \(S_n\).
\end{proof}

Building off this result, we can sum over all the topdrop-valid necklaces to count the orbits of a particular size.

\begin{theorem}\label{thm:orbit_counting}
For every natural number \(n\) and natural number \(k\), the number of orbits of size \(k\) in \(S_n\) is \[\sum_{\{\mathcal{N} \in A_n \,:\, |\mathcal{N}|=k\}}{\frac{\left(n-\left(\textnormal{\# of distinct values in }\mathcal{N}\right)\right)!}{P(\mathcal{N})}}\]
\end{theorem}

\begin{proof}
Each orbit has one necklace, so summing the orbit counts from all necklaces gives the total number of orbits.
\end{proof}

We use these theorems to establish orbit counts.

\begin{theorem}\label{thm:orbit_counts}
The following is true of the topdrop map in \(S_n\):
\begin{enumerate}
    \item There are \((n-2)!\) orbits of size two, when \(n \geq 2\)
    \item There are no orbits of size three
    \item There are \((n-2)!\) orbits of size four, when \(n \geq 3\)
    \item The number of orbits of size five, when \(n \geq 7\), is 
        \begin{itemize}
        \item \((n-6)(n-5)!\) orbits of size five if \(n\) is even
        \item \((n-5)(n-5)!\) orbits of size five if \(n\) is odd
        \end{itemize}
    \item The number of orbits of size six, when \(n \geq 4\), is at least        \begin{itemize}
        \item \((2n^2-11n+14)(n-4)!\) orbits of size six when \(n\) is even
        \item \((2n^2-12n+18)(n-4)!\) orbits of size six when \(n\) is odd
        \end{itemize}
    \item The number of orbits of size eight, when \(n \geq 5\), is at least        \begin{itemize}
        \item \(\left( \frac{3}{2}n^3 - \frac{35}{2}n^2 + 64n - 70 \right) (n - 5)!\) if \(n\) is congruent to \(0\) or \(2\) modulo \(6\)
        \item \(\left( \frac{3}{2}n^3 - \frac{35}{2}n^2+ \frac{137}{2}n - \frac{181}{2} \right) (n - 5)!\) if \(n\) is congruent to \(1\) modulo \(6\)
        \item \(\left( \frac{3}{2}n^3 - \frac{35}{2}n^2 + \frac{135}{2}n - \frac{171}{2} \right) (n-5)!\) if \(n\) is congruent to \(3\) or \(5\) modulo \(6\)
        \item \(\left( \frac{3}{2}n^3 - \frac{35}{2}n^2 + 66n - 80 \right) (n-5)!\) if \(n\) is congruent to \(4\) modulo \(6\)
        \end{itemize}
\end{enumerate}
\end{theorem}

\begin{remark}
There are no orbits of size seven in \(S_n\) for \(2 \leq n \leq 14\). It is unknown whether there are orbits of size seven for larger \(n\).
\end{remark}

\begin{remark}
While the true number of orbits of sizes six and eight is generally greater than these lower bounds, they are not far off, at least through \(S_{14}\), the highest experimentally verified permutation length. In \(S_14\), the lower bound gives 914,457,600 orbits of size six when in reality there are 914,538,240 orbits; the lower bound gives 548,674,560 orbits of size eight when in reality there are 548,691,840 such orbits.
\end{remark}

The proofs of parts (1) and (3)-(6) are each proceeded by one or more lemmas that make statements about the topdrop-validity of certain necklaces. The part of the proof which follows the lemma(s) counts these necklaces and uses Theorem \ref{thm:orbit_counting} to make a statement about the number of orbits. The proof of part (2) stands by itself.

\begin{lemma}\label{lem:neck_two}
A necklace of size two is topdrop-valid in \(S_n\) if and only if it is equivalent to \([n-1, n]\), \(n \geq 2\).
\end{lemma}

\begin{proof}
 When \(n\) is \(2\), there is one orbit in \(S_2\) of size two, and it has the necklace \([1,2] = [n-1, n]\).

Assume \(n \geq 3\).  First, we verify that all necklaces of the form \([n-1, n]\) are topdrop-valid. We must show that there is some \(\pi\) so that \(\mathcal{N}(\pi) = [n-1, n]\). We claim that \(\pi = (n-1)\seg{1}{>}{n-2}{n-2}(n)\) works where the value of \(\pi_i\) is arbitrary for \(1 \leq i \leq n-2\). Indeed,
\[T(\pi) = (n)\seg{n-2}{<}{1}{n-2}(n-1)\]
\[T^2(\pi) = (n-1)\seg{1}{>}{n-2}{n-2}(n) = \pi.\]

Second, we show that any topdrop-valid necklace of size two is equivalent to \([n-1, n]\). Let \([a,b]\) be a necklace of size two for \(S_n\). Any orbit with necklace \([a,b]\) must contain a permutation \(\pi\) so that \(\pi_1 = a\) and \(T(\pi)_1 = b\). In an orbit of size two, \(T(\pi)_1 = T^{-1}(\pi)_1 = b\), which implies that \(\pi_n = b\). So \(\pi = a \seg{2}{>}{n-1}{n-2} b\), and it is clear from the definition of the topdrop map that we must have \(a = n\) or \(a = n-1\) so that \(T(\pi)_1 = b\). Hence \(T(\pi) = b \seg{n-1}{<}{2}{n-2} a\), and similarly we must have \(b = n\) or \(b = n-1\), but \(b\) must not equal \(a\). Hence \([a,b]\) must be equivalent to \([n-1,n]\).
\end{proof}

\begin{proof}[Proof of part (1)]
By Lemma \ref{lem:neck_two}, there is only one term in the sum of Theorem \ref{thm:orbit_counting} for orbits of size two, which is for the necklace \([n-1, n]\). The fundamental period of the necklace is one, and there are always two different values in the necklace, so there are \(\frac{(n-2)!}{1}\) orbits of size two.
\end{proof}

\begin{proof}[Proof of part (2)]
Assume for contradiction that there is a natural number \(n\) and a \(\pi\) in \(S_n\) so that \(\mathcal{O}(\pi)\) is an orbit of size three. Let \(a = \pi_1\), \(b = T(\pi)_1\), and \(c = T^2(\pi)_1\). We cannot have \(a = n\), or else \(b = c\). Then we must have \begin{align*}
\pi &= a \seg{2}{>}{a}{a-1} b \seg{a+1}{>}{n-1}{n-a-1} c \\
T(\pi) &= b \seg{a+1}{>}{n-1}{n-a-1}{c}\seg{a}{<}{2}{a-1} a \\
T^2(\pi) &= {c}\seg{a}{<}{2}{a-1} a \seg{n-1}{<}{a+1}{n-a-1} b
\end{align*} Since \(T^3(\pi) = \pi\), and \(a\) needs to be at the front of \(\pi\), we find \(c = 1+(a+1-2) = a\), which is impossible.
\end{proof}

\begin{lemma}\label{lem:neck_four}
A necklace of size four is topdrop-valid in \(S_n\) if and only if it is equivalent to \([a, n-1,  a, n]\) for some natural number \(a\) with \(1 \leq a \leq n-2\).
\end{lemma}
\begin{proof}
Assume that \([a, b, c, d]\) is a topdrop-valid necklace of size four for \(S_n\). We will split into cases depending on the values of \(a\), \(b\), \(c\), and \(d\) and show that all cases which match the criteria of the proposition work while all other cases produce contradictions. First, if \(a\), \(b\), \(c\), and \(d\) each equal \(n\) or \(n-1\), then \(T^2(\pi) = \pi\) for any \(\pi\) in the orbit and the orbit size would not be four, so assume without loss of generality that \(1 \leq a \leq n-2\).

Let \(\pi = \seg{1}{>}{n}{n}\) be a permutation whose orbit has size four with the necklace \([a, b, c, d]\). Then \(a = \pi_1\). So \[T(\pi) = \seg{a+1}{>}{n}{n-a}\seg{a}{<}{1}{a}\] This implies \(b = \pi_{\pi_1+1}\). In the case that \(b < n-a\), then \[T^2(\pi) = \seg{a+b+1}{>}{n}{n-a-b} \seg{a}{<}{1}{a} \seg{a+b}{<}{a+1}{b},\] and \(c = \pi_{a+b+1}\). Since \(\pi_n\) must be at the front of \(T^{-1}(\pi) = T^{3}(\pi)\), \(c = n-(a+b+1)\), and \[T^3(\pi) = \pi_n \seg{a}{<}{1}{a} \seg{a+b}{<}{a+1}{b} \seg{n-1}{<}{a+b+1}{n-(a+b+1)}.\] Since \(T^{4}(\pi) = \pi,\) \(\pi_1\) must be at the front of \(T^{4}(\pi)\), so \(d = \pi_n = a-1\) meaning \[T^{4}(\pi) = \seg{a+b}{<}{a+1}{b} \seg{n-1}{<}{a+b+1}{n-(a+b+1)} \seg{1}{>}{a}{a} \pi_n.\] But this does not equal \(\pi = \seg{1}{>}{n}{n}\), so we have a contradiction if \(b < n-a\).

In the case that \(b = n-a\), we have \[T^2(\pi) = \seg{a}{<}{1}{a} \seg{n}{<}{a+1}{n-a}\] and \(c = \pi_a\). Since \(\pi_n\) must be at the front of \(T^{-1}(\pi) = T^{3}(\pi)\), \(c = a\), and \[T^3(\pi) = \seg{n}{<}{a+1}{n-a} \seg{1}{>}{a}{a}.\] Since \(T^{4}(\pi) = \pi,\) \(\pi_1\) must be at the front of \(T^{4}(\pi)\). If \(a \neq 1\), then \(d = n-a\). But also \(b=n-a\), and the only way this is possible is if \(c\) equals \(n\) or \(n-1\), which is impossible since \(c = a \leq n-2\). If \(a = 1\), then \(d = n-1\) or \(d = n\).  The case \(d = n-1\) is impossible for the same reason that \(a \neq 1\) is impossible. So we have \(a = 1\), \(b = n-1\), \(c = a\), and \(d = n\). So \[T^{3}(\pi) = \seg{n}{<}{1}{n}\] and from \(d = \pi_n = n\) we find that \[T^4(\pi) = \seg{1}{>}{n}{n} = \pi.\]

If \(b > n-a\) but \(b \leq n-2\), we have \[T^2(\pi) = \seg{n-b}{<}{1}{n-b} \seg{n-b+1}{>}{a}{a+b-n} \seg{n}{<}{a+1}{n-a}\] and \(c = \pi_{n-b}\). Since \(\pi_n\) must be at the front of \(T^{-1}(\pi) = T^{3}(\pi)\), \[c = (n-b)+(a-n+b) = a,\] but this would imply that \(b\) equals \(n\) or \(n-1\).

If \(b > n-a\) and \(b=n\) or \(b = n-1\), we have that \[T^2(\pi) = \seg{1}{>}{a}{a} \seg{n}{<}{a+1}{n-a}\] and \(c = \pi_1 = a\). Then \[T^{3}(\pi) = \seg{n}{<}{1}{n},\] and \(d = \pi_n\) equals \(n\) or \(n-1\). If \(b = n\), then \(d = n-1\), or else \(a\) and \(c\) would each have to equal \(n\) or \(n-1\). Similarly if \(b = n-1\) then \(d = n\). Under these conditions we find that \[T^{4}(\pi) = \seg{1}{>}{n}{n}.\]

From the cases which work without contradictions we find that topdrop-valid necklaces of size four are the necklaces equivalent to \([a, n-1, a, n\)] or \([a, n, a, n-1]\) with \(1 \leq a \leq n-2\), which reduces to just \([a, n-1, a, n]\) since \([a, n, a, n-1]\) is equivalent.
\end{proof}

\begin{proof}[Proof of part (3)]
Since \(1 \leq a \leq n-2\) in the previous proposition, there are \(n-2\) different necklaces in the sum of Theorem \ref{thm:orbit_counting}. Each necklace has \((n-3)\) distinct values, and the fundamental period of each necklace is one. So there are \((n-2)\frac{(n-3)!}{1} = (n-2)!\) orbits of size four. 
\end{proof}

\begin{lemma}\label{lem:neck_five}
A necklace of size five is topdrop-valid in \(S_n\) if and only if it is equivalent to \([1,a,b,c,d]\) where \(a+b = n\), \(c+d=n\), \(c=a+1\), and \(1 \leq a,b,c,d \leq n-2\) where \(a\), \(b\), \(c\) and \(d\) are all different natural numbers.
\end{lemma}

\begin{proof}
Similar to the proof of the proposition for necklaces of size four, we assume that \([p,q,r,s,t]\) is a topdrop-valid necklace of size five in \(S_n\), split into cases depending on the values of \(p\), \(q\), \(r\), \(s\), and \(t\) and show that all cases which match the criteria of the proposition work while all other cases produce contradictions. We write the necklace with the letters \(p\) through \(t\) instead of \(a\) through \(e\) because each variable from \(p\) through \(t\) may take a different role depending on the case, for example which variable equals \(1\) in the necklace.

Let \(\pi = \seg{1}{>}{n}{n}\) be a permutation whose orbit has size five with the topdrop-necklace \([p,q,r,s,t]\). Because the necklace size is odd, we can safely assume that \(p,q,r,s,t \leq n-2\) because of Theorem \ref{thm:necklaces_with_n_or_n_minus_one_special_properties}. We can also assume that they are all different natural numbers, because if any two were equal, then one would have to equal \(n\) or \(n-1\). Let \(p = \pi_1\), so that \[T(\pi) = \seg{p+1}{>}{n}{n-p} \seg{p}{<}{1}{p}.\]

First, suppose \(q < n-p\). Then \[T^2(\pi) = \seg{p+q+1}{>}{n}{n-p-q} \seg{p}{<}{1}{p} \seg{p+q}{<}{p+1}{q}.\] We split into three further cases: \(r < n-p-q\), \(n-p-q \leq r < n-q\), and \(r \geq n-q\).

Suppose \(r < n-p-q\).  Then \[T^3(\pi) = \seg{p+q+r+1}{>}{n-p-q-r}{} \seg{p}{<}{1}{p} \seg{p+q}{<}{p+1}{q} \seg{p+q+r}{<}{p+q+1}{r}.\] We know \[T^{4}(\pi)_1 = T^{-1}(\pi)_1 = \pi_n,\] so \[s = \pi_{p+q+r+1} = n-(p+q+r+1),\] and \[T^4(\pi) = \pi_n \seg{p}{<}{1}{p} \seg{p+q}{<}{p+1}{q} \seg{p+q+r}{<}{p+q+1}{r} \seg{n-1}{<}{p+q+r+1}{n-(p+q+r+1)}.\] Since \(T^{5}(\pi)_1 = \pi_1\), \[t = \pi_n = 1+(p-1) = p,\] but this is impossible because \(p\) and \(t\) are adjacent in the necklace, so this case is contradictory.

Suppose \(n-p-q \leq r < n-q\). Then \[T^3(\pi) = \seg{n-q-r}{<}{1}{n-q-r} \seg{p+q}{<}{p+1}{q} \seg{n-q-r+1}{>}{p}{p+q+r-n} \seg{n}{<}{p+q+1}{n-p-q}.\] Since \(T^4(\pi)_1\ = \pi_n\), \(d = \pi_{n-b-c} = p+q\), so \[T^4(\pi) = \seg{n}{<}{p+q+1}{n-p-q} \seg{p}{<}{n-q-r+1}{p-n+q+r} \seg{p+1}{>}{p+q}{q} \seg{1}{>}{n-q-r}{n-q-r}.\] Since we assumed \(q < n-p\), we can be sure that \(\seg{n}{<}{p+q+1}{n-p-q}\) is not empty and that segment is thus where \(\pi_n\) resides. Then \(t = \pi_n =  q+r\), so that \(T^{5}(\pi)_1 = \pi_1\). Then \[T^{5}(\pi) = \seg{1}{>}{n-q-r}{n-q-r} \seg{p+q}{<}{p+1}{q} \seg{n-q-r+1}{>}{p}{p+q+r-n} \seg{p+q+1}{>}{n}{n-p-q}.\] Since we assumed \(r < n-q\), we can be sure that \(\seg{1}{>}{n-q-r}{n-q-r}\) is not empty and that is where \(\pi_1\) resides. We need that \[\pi = \seg{1}{>}{n}{n} = T^5(\pi).\] The segment \(\seg{p+q}{<}{p+1}{q}\) must either be empty or contain one element, so \(p+q \leq p+1\), which implies that \(q = 1\), and that the segment is not empty. Since \(\pi_{p+1}\) must immediately precede \(\pi_{p+2}\), we deduce that \(\seg{n-q-r+1}{>}{p}{p+q+r-n}\) is empty. Since \(\seg{1}{>}{n-q-r}{n-q-r}\) is not empty, we need that \(n-q-r+1 = p+q\), which implies \(n-1 = p + r\). Since \(s = p + q = p +1\), we deduce that \(n = r + s\). Since also \(t = r + 1\), we deduce that \(n = p + t\). So when \(q = 1\), \(n = r + s\), and \(n = p + t\), which are the criteria of the proposition, then \(T^5(\pi) = \pi\).

Suppose \(r \geq n-q\). Then \[T^3(\pi) = \seg{n-r+p}{<}{p+1}{n-r} \seg{n-r+p+1}{>}{p+q}{q+r-n} \seg{1}{>}{a}{a} \seg{n}{<}{p+q+1}{n-p-q}.\] Since \(T^4(\pi)_1 = \pi_n\), \(s = p+q\), and \[T^4(\pi) = \seg{n}{<}{p+q+1}{n-p-q} \seg{p}{<}{1}{p} \seg{p+q}{<}{n-r+p+1}{q+r-n} \seg{p+1}{>}{n-r+p}{n-r}.\] Note that since we assumed \(q < n-p\), we can be sure that \(\seg{n}{<}{p+q+1}{n-p-q}\) is not empty and that is where \(\pi_n\) resides. Since \(T^{5}(\pi)_1 = \pi_1\), \(t = n-q-1\), and \[T^{5}(\pi) = \pi_1 \seg{p+q}{<}{n-r+p+1}{q+r-n} \seg{p+1}{>}{n-r+p}{n-r} \seg{2}{>}{p}{p-1} \seg{p+q+1}{>}{n}{n-p-q}.\] Note that we can be sure that \(\seg{p}{<}{1}{p}\) is non-empty, so that is is valid to single out \(\pi_1\) from that segment in this way. Since \(p\) and \(q\) cannot both be one, \(p+q \geq 3\). Since \(\pi_1\) must immediately precede \(\pi_2\), \(\seg{p+q}{<}{n-r+p+1}{q+r-n}\) must be empty. Since \(\seg{p+1}{>}{n-r+p}{n-r}\) is the only segment which can include \(\pi_{p+1}\), which must appear somewhere in \(T^5(\pi)\), \(\seg{p+1}{>}{n-r+p}{n-r}\) is not empty. So \(\pi_1\) immediately precedes \(\pi_{p+1}\), which implies \(p = 1\). So \(\seg{2}{>}{p}{p-1}\) is empty, which implies \(p+q+1=n-r+p+1\) from which we deduce \(q+r = n\). Since \(s = p+q = 1 + q\), and \(t = n-q-1\), we also have \(n = s + t\). So \(p = 1\), \(q+r=n\), and \(s+t=n\), matching the criteria of the proposition, and \([p,q,r,s,t]\) is topdrop-valid whenever these criteria are met. This concludes the analysis of cases where \(q < n-p\).

Second, suppose \(q \geq n-p\). Then \[T^2(\pi) = \seg{n-q}{<}{1}{n-q} \seg{n-q+1}{>}{a}{p+q-n} \seg{n}{<}{p+1}{n-p}.\] We split again into three cases: \(r < n-q\), \(n-q \leq r < p\), and \(r \geq p\).

Suppose \(r < n-q\). Then \[T^3(\pi) = \seg{n-q-r}{<}{1}{n-q-r} \seg{n-q+1}{>}{p}{p+q-n} \seg{n}{<}{p+1}{n-p} \seg{n-q-r+1}{>}{n-q}{r}.\] Since \(T^4(\pi)_1 = \pi_n\), we deduce \(s = p-r\) and \[T^4(\pi) = \seg{n}{<}{p+1}{n-p} \seg{n-q-r+1}{>}{n-q}{r} \seg{p}{<}{n-q+1}{p+q-n} \seg{1}{>}{n-q-r}{n-q-r}.\] Note that we assume \(p \leq n-2\), we can be sure that \(\seg{n}{<}{p+1}{n-p}\) is not empty and that is where \(\pi_n\) resides. So \(t = q + r\) and \[T^5(\pi) = \seg{1}{>}{n-q-r}{n-q-r} \seg{n-q+1}{>}{p}{p+q-n} \seg{n-q}{<}{n-q-r+1}{r} \seg{p+1}{>}{n}{n-p}.\] Note that since we assume \(r < n-q\), \(\seg{1}{>}{n-q-r}{n-q-r}\) is not empty, and we can be sure that is where \(\pi_1\) resides. Since \(r \geq 1\), \(\seg{n-q}{<}{n-q-r+1}{r}\) cannot be empty, which means that \[n-q=n-q-r+1,\] which implies that \(r = 1\). Since the following segment is \(\seg{p+1}{>}{n}{n-p}\), \(n-q+1=p+1\), so \(n=p+q\). Since \(s = p-r = p-1\) and \(t = q+r = q+1\), we also have \(n=s+t\). So \(r = 1\), \(s+t=n\), and \(p+q=n\), matching the criteria of the proposition, and \([p,q,r,s,t]\) is topdrop-valid whenever these criteria are met.

Suppose \(n - q \leq r < p\). Then \[T^3(\pi) = \seg{r+1}{>}{p}{p-r} \seg{n}{<}{p+1}{n-p} \seg{n-q+1}{>}{r}{q+r-n} \seg{1}{>}{n-q}{n-q}.\] Since \(T^4(\pi)_1 = \pi_n\), \(s = p - r\) and \[T^4(\pi) = \seg{n}{<}{p+1}{n-p} \seg{n-q+1}{>}{r}{r+q-n} \seg{1}{>}{n-q}{n-q} \seg{p}{<}{r+1}{p-r}.\] Note that we assume \(p \leq n-2\), we can be sure that \(\seg{n}{<}{p+1}{n-p}\) is not empty and that is where \(\pi_n\) resides. Then \(t = q+r-p\), and \[T^5(\pi) = \seg{1}{>}{n-q}{n-q} \seg{p}{<}{r+1}{p-r} \seg{r}{<}{n-q+1}{r+q-n} \seg{p+1}{>}{n}{n-p}.\] Note that \(\seg{1}{>}{n-q}{n-q}\) is not empty since we assumed that \(q \leq n-2\), so we can be sure that \(\pi_1\) resides in that segment. Since \(\seg{p+1}{>}{n}{n-p}\) can contain at most one element, \(p \leq r + 1\). We assumed \(p > r\), which implies \(p = r + 1\). Since \(s = p-r\), we deduce \(s=1\). Since \(\pi_{r+1} = \pi_p\) must be immediately followed by \(\pi_{p+1}\), the segment \(\seg{r}{<}{n-q+1}{r+q-n}\) must be empty, which implies \(r < n-q+1\). Since we assumed \(r \geq n-q\), we deduce \(r = n-q\), so \(q+r=n\). Since \(t = q + r - p\), we find that \(p + t = q + r = n\). So \(s=1\), \(q+r=n\), and \(s+t=n\), matching the criteria of the proposition, and \([p,q,r,s,t]\) is topdrop-valid whenever these criteria are met.

Suppose \(r \geq p\). Since \(q\) is not equal to \(n\) or \(n-1\), \(r\) is not equal to \(p\), so we can improve this inequality to \(r > p\). Then \[T^3(\pi) = \seg{n+p-r}{<}{p+1}{n-r} \seg{n+p-r+1}{>}{n}{r-p} \seg{p}{<}{n-q+1}{p+q-n} \seg{1}{>}{n-q}{n-q}.\] Since \(T^4(\pi) = \pi_n\), \(s = n-p-1\), and \[T^4(\pi) = \pi_n \seg{p}{<}{n-q+1}{p+q-n} \seg{1}{>}{n-q}{n-q} \seg{n-1}{<}{n+p-r+1}{r-p-1} \seg{p+1}{>}{n+p-r}{n-r}.\] Note that \(\seg{n+p-r+1}{>}{n}{r-p}\) is not empty since \(r > p\), so we can be sure \(\pi_n\) resides in that segment and it is valid to extract \(\pi_n\) as we have from it. Then \(t = p + q + 1 - n\), and \[T^5(\pi) =\seg{1}{>}{n-q}{n-q} \seg{n-1}{<}{n+p-r+1}{r-p-1} \seg{p+1}{>}{n+p-r}{n-r}  \seg{n-q+1}{>}{p}{p+q-n}  \pi_n.\] We note again that \(\seg{1}{>}{n-q}{n-q}\) is not empty since we assumed that \(q \leq n-2\), so we can be sure that \(\pi_1\) resides in that segment. If \(\seg{n-q+1}{>}{p}{p-n+q}\) were non-empty, then \(p = n-1\). Since \(s + p = n-1\), and \(s \geq 1\), this is impossible. So \(\seg{n-q+1}{>}{p}{p-n+q}\) is empty. So \(p < n-q+1\), implying \(p + q \leq n\). Since we assumed \(q \geq n-p\), we deduce \(p + q = n\). Since \(t = p + q + 1 - n\), we find \(t = 1\). If \(\seg{p+1}{>}{n+p-r}{n-r}\) were empty, then \(n + p - r < p + 1\), so \(n-r < 1\), which would imply \(r = n\). This would necessitate \(q=s\), so \(p+q=n\) would imply \(p+s=n\). But \(s = n-p-1\), which implies \(p+s=n-1\). So \(\seg{p+1}{>}{n+p-r}{n-r}\) is not empty. Because \(\pi_{n+p-r}\) is followed immediately by \(\pi_n\), we deduce \(n+p-r+1=n\), which implies \(p = r-1\). Together with \(s = n-p-1\), this implies that \(r+s = n\). So we have that \(t=1\), \(p+q=n\), and \(r+s=n\), matching the criteria of the proposition, and \([p,q,r,s,t]\) is topdrop-valid whenever these criteria are met.
\end{proof}

\begin{proof}[Proof of part (4)]
We want to determine the sum of Theorem \ref{thm:orbit_counting}, and the previous lemma determines precisely which necklaces have terms in that sum. When \(n \leq 6\), there do not exist integers \(a\), \(b\), \(c\), and \(d\) fitting the criteria of the proposition.

Suppose \(n\) is odd, and \(n \geq 7\). \(a\) cannot equal \(1\), since it comes immediately after \(1\) in the necklace. 

If \(2 \leq a \leq (n-3)/2\), then \(b=n-a\) is distinct from \(a\) since \(n\) is odd so that \(a \neq n-a\), \(c = a+1\) is distinct from \(a\), and \(c\) is distinct from \(b\) since \(c=b\) would imply \(n=2a+1\), but \(2a+1 \leq 2(n-3)/2+1 =  n-2\). \(d = n-c = n-a-1\) is distinct from \(a\) because otherwise \(n=2a+1\), \(d\) is distinct from \(b\) because \(n-a \neq n-a-1\), and \(d\) is distinct from \(c\) because \(n\) is odd. Since \(2 \leq a \leq (n-3)/2\) also guarantees that  \(1 \leq a,b,c,d \leq n\), there is one topdrop-valid necklace of size five for every value of \(a\) from \(2\) through \((n-3)/2\), where the values of \(b\), \(c\), and \(d\) are determined by \(a\).

If \(a = (n-3)/2+1\), then \[b = n-a = n-(n-3)/2-1 = (n-3)/2+2,\] but also \(c = a+1 = (n-3)/2+2\), which violates that \(b\) and \(c\) must be distinct.

If \((n-3)/2+2 \leq a \leq n-3\), then \(b = n-a\) is distinct from \(a\) since \(n\) is odd so that \(a \neq n-a\). \(c=a+1\) is distinct from \(a\), and \(c\) is distinct from \(b\) because otherwise \(n = 2a+1\), but \(2a+1 \geq 2((n-3)/2+2) = n+1\). \(d = n-c = n-a-1\) is distinct from \(a\) because otherwise \(n=2a+1\), \(d\) is distinct from \(b\) because \(n-a \neq n-a-1\), and \(d\) is distinct from \(c\) because \(n\) is odd. Since \((n-3)/2+2 \leq a \leq n-3\) also guarantees that \(1 \leq a,b,c,d \leq n\), there is one topdrop-valid necklace of size five for every value of \(a\) from \((n-3)/2+2\) through \(n-3\).

As stated in the proof of the proposition, we know \(a\) is not equal to \(n\) or \(n-1\). If \(a = n-2\), then \(d = 1\), but then \(1\) would follow \(d=1\) in the necklace, which is impossible. Since \(a \leq n\), we have covered all cases.

There are \(n-5\) values of \(a\) so that \(2 \leq a \leq (n-3)/2\) or \((n-3)/2+2 \leq a \leq n-3\), and there is one topdrop-valid necklace with five distinct values for each value. The fundamental period of each necklace is one. Thus according to Theorem \ref{thm:orbit_counting}, there are \((n-5)(n-5)!\) orbits of size five when \(n\) is odd.

Suppose \(n\) is even, and \(n \geq 8\). Again, \(a\) cannot equal \(1\), since it comes immediately after \(1\) in the necklace.

If \(2 \leq a \leq n/2-2\), then \(b=n-a\) is distinct from \(a\), since if \(b = a\) then \(n=2a\) so \(a = n/2>n/2-2\). \(c = a+1\) is distinct from \(a\), and \(c\) is distinct from \(b\) since \(b=c\) would imply \(n=2a+1\), but \(n\) is even. \(d = n-c=n-a-1\) is distinct from \(a\) since \(n \neq 2a+1\), \(d\) is distinct from \(b\) since \(n-a\neq n-a-1\), and \(d\) is distinct from \(a\) since \(n\neq2a+1\). Since \(2 \leq a \leq n/2-2\) also guarantees that \(1 \leq a,b,c,d \leq n\), there is one topdrop-valid necklace of size five for every value of \(a\) from \(2\) through \(n/2-2\).

If \(a = n/2-1\), then \(c = n/2\), so \(d = n-n/2 = n/2\), so \(c\) and \(d\) are not distinct.

If \(a = n/2\), then \(b = n-a = n/2\), so \(a\) and \(b\) are not distinct.

If \(n/2+1 \leq a \leq n-3\), then \(b=n-a\) is distinct from \(a\), since otherwise \(n=2a\), so \(a = n/2 < n/2+1\). \(c=a+1\) is distinct from \(a\), and \(c\) is distinct from \(b\) because otherwise \(n=2a+1\) and \(n\) would be odd. \(d=n-c=n-a-1\) is distinct from \(c\) because otherwise \(n=2a+1\), \(d\) is distinct from \(b\) since \(n-a\neq n-a-1\), and \(d\) is distinct from \(a\) because otherwise \(n=2a+1\). \(n/2+1 \leq a \leq n-3\) also guarantees that \(1 \leq a,b,c,d \leq n\), so there is one topdrop-valid necklace of size five for every value of \(a\) from \(n/2+1\) through \(n-3\).

As in the case where \(n\) is odd, \(a\) must not equal \(n\), \(n-1\), or \(n-2\).

There are \(n-6\) values of \(a\) where \(2 \leq a \leq n/2-2\) or \(n/2+1 \leq a \leq n-3\), and there is one topdrop-valid necklace with five distinct values for each value. The fundamental period of each necklace is one. Thus according to Theorem \ref{thm:orbit_counting}, there are \((n-6)(n-5)!\) orbits of size five when \(n\) is odd.
\end{proof}

\begin{lemma}
Necklaces of the form \([a,b,n,b,a,n]\) or \([a,b,n-1,b,a,n-1]\) with \(1 \leq a,b \leq n-2\), \(a+b = n-1\), and \(a\) and \(b\) distinct are topdrop-valid in \(S_n\), where \(n \geq 4\). Also topdrop-valid are necklaces of the form \([a,b,n-1,b,a,n]\) with \(1 \leq a,b \leq n-2\), \(a+b \neq n-1\), and \(a\) and \(b\) distinct.
\end{lemma}
\begin{proof}
First, we show that necklaces of the form \([a,b,n,b,a,n]\) with \(1 \leq a,b \leq n-2\) and \(a+b = n-1\) are topdrop-valid. The proof for necklaces of the form \([a,b,n-1,b,a,n-1]\) is identical.

Let \(\pi\) be a permutation in \(S_n\) with \(\pi_1 = a\), \(\pi_{a+1} = b\), and \(\pi_n = n\), and assume \(1 \leq a,b \leq n-2\) and \(a+b=n-1\). Then \(T(\pi) = \seg{a+1}{>}{n}{n-a} \seg{a}{<}{1}{a}\), and \(T^2(\pi) = \pi_n \seg{a}{<}{1}{a} \seg{n-1}{<}{a+1}{n-a-1}\) because \(\pi_{a+1} = b = n-a-1\). Since \(T^{-1}(\pi)_1 = \pi_n = n\), we know that "\(n\) \(a\) \(b\) \(n\)" is part of the necklace, and by Theorem \ref{thm:necklaces_with_n_or_n_minus_one_special_properties} we deduce that the full necklace is \([a,b,n,b,a,n]\).

Second, we show that necklaces of the form \([a,b,n-1,b,a,n]\) with \(1 \leq a,b \leq n-2\) and \(a+b \neq n-1\) are topdrop-valid. Let \(\pi\) be a permutation in \(S_n\) with \(\pi_1 = a\), \(\pi_{a+1} = b\), and \(\pi_n = n\). Assume \(1 \leq a,b \leq n-2\) and \(a+b\neq n-1\). Then \(T(\pi) = \seg{a+1}{>}{n}{n-a} \seg{a}{<}{1}{a}\), and since now we assume that \(a+b \neq n-1\), \(T^2(\pi)_1\) cannot be \(\pi_n\). Nor can \(T^2(\pi)_1\) be \(a\) or \(b\), but it may be any other possible value, including \(n-1\). Letting \(T^2(\pi)_1\) be \(n-1\), and remembering \(T^{-1}(\pi)_1 = \pi_n\), we find that "\(n\) \(a\) \(b\) \(n-1\)" is part of the necklace. By Theorem \ref{thm:necklaces_with_n_or_n_minus_one_special_properties} the full necklace is \([a,b,n-1,b,a,n]\).
\end{proof}

\begin{remark}
For \(n \leq 12\), all of the necklaces of size six for \(S_n\) have one of the above forms. For \(n\) greater than \(12\) this is no longer the case. For example, \([2,5,6,3,4,7]\) is valid in \(S_{13}\).
\end{remark}

\begin{proof}[Proof of part (5)]
First, we count necklaces of the form \([a,b,n,b,a,n]\) or \([a,b,n-1,b,a,n-1]\). When \(n\) is even, there are \(\frac{n-2}{2}\) ways to choose two distinct numbers that sum to \(n-1\) from the natural numbers \(1\) through \(n\). So there are \(\frac{n-2}{2}\) possible pairs of \(a\) and \(b\) for each of the two types of necklaces, and thus in all \(n-2\) necklaces . When \(n\) is odd, there are \(\frac{n-3}{2}\) ways to choose two numbers summing to \(n-1\), and \(n-3\) necklaces in all. There are three different values in each necklace, and the fundamental period of each necklace is one.

Second, we count necklaces of the form \([a,b,n-1,b,a,n]\) where \(a\) and \(b\) do not sum to \(n-1\). If we ignore the restriction on the sum, there are \((n-2)(n-3)\) choices for \(a\) and \(b\). When \(n\) is even, all but \(n-2\) of these do not sum to \(n-1\), and when \(n\) is odd all but \(n-3\) of these do not sum to \(n-1\). There are four different values in each necklace, and the fundamental period of each necklace is one.

By Theorem \ref{thm:orbit_counting}, we find that when \(n\) is even there are \([(n-2)(n-3)-(n-2)](n-4)!+(n-2)(n-3)!\) orbits corresponding to these forms, and when \(n\) is odd there are \([(n-2)(n-3)-(n-3)](n-4)!+(n-3)(n-3)!\) orbits corresponding to these forms. These counts simplify to the given formulas.
\end{proof}

\begin{lemma}
Necklaces of the form \([a,b,c,n,c,b,a,n]\) or \([a,b,c,n-1,c,b,a,n-1]\) with \(1 \leq a,b,c \leq n-2\), \(a+b+c = n-1\), and \(a\), \(b\), and \(c\) distinct are topdrop-valid in \(S_n\), where \(n \geq 4\). Also topdrop-valid are necklaces of the form \([a,b,c,n-1,c,b,a,n]\) with \(1 \leq a,b,c \leq n-2\), \(a+b+c \neq n-1\), \(a+b\neq n-1\), \(b+c\neq n-1\), and \(a\), \(b\), and \(c\) distinct.
\end{lemma}

\begin{proof}
First, we show that necklaces of the form \([a,b,c,n,c,b,a,n]\) with \(1 \leq a,b,c \leq n-2\) and \(a+b+c = n-1\) are topdrop-valid. The proof for necklaces of the form \([a,b,c,n-1,c,b,a,n-1]\) is identical. We show that it is always possible to construct a permutation with this necklace. Let \(\pi\) be a permutation in \(S_n\) with \(\pi_1 = a\), \(\pi_{a+1} = b\), \(\pi_{a+b+1} = c\) and \(\pi_n = n\), and assume \(1 \leq a,b,c \leq n-2\),  \(a+b+c=n-1\), and \(a\), \(b\), and \(c\) are distinct. Then \(T(\pi) = \seg{a+1}{>}{n}{n-a} \seg{a}{<}{1}{a}\). Since \(b = n-a-c-1 < n-a\), \(T^2(\pi) = \seg{a+b+1}{>}{n}{n-a-b} \seg{a}{<}{1}{a} \seg{a+b}{<}{a+1}{b}\). Since \(c = n-a-b-1\), \(T^3(\pi) = \pi_n \seg{a}{<}{1}{a} \seg{a+b}{<}{a+1}{b} \seg{n-1}{<}{a+b+1}{c}\). Since \(T^{-1}(\pi)_1 = \pi_n = n\), we know that "\(n\) \(a\) \(b\) \(c\) \(n\)" is part of the necklace, and by Theorem \ref{thm:necklaces_with_n_or_n_minus_one_special_properties} the full necklace is \([a,b,c,n,c,b,a,n]\).

Second, we show that necklaces of the form \([a,b,c,n-1,c,b,a,n]\) with \(1 \leq a,b,c \leq n-2\), \(a+b+c \neq n-1\), \(a+b\neq n-1\), \(b+c\neq n-1\), and \(a\), \(b\), \(c\) distinct are topdrop-valid. Again, we show how to construct a permutation with this necklace. Let \(\pi\) be a permutation in \(S_n\) with \(\pi_1 = a\), \(\pi_{a+1} = b\), and \(\pi_n = n\). Assume \(1 \leq a,b,c \leq n-2\), \(a+b\neq n-1\), \(b+c\neq n-1\), \(a+b+c\neq n-1\), and that \(a\), \(b\), and \(c\) are distinct. Then \(T(\pi) = \seg{a+1}{>}{n}{n-a} \seg{a}{<}{1}{a}\). We consider both \(b<n-a-1\) or \(b>n-a-1\) and show that the necklace is topdrop-valid in either case.

If \(b<n-a-1\), then \(T^2(\pi) = \seg{a+b+1}{>}{n}{n-a-b} \seg{a}{<}{1}{a} \seg{a+b}{<}{a+1}{b}\). Let \(\pi_{a+b+1}=c\), which is valid since we have not yet specified \(\pi_{a+b+1}\), while \(\pi_{a+b+1}\) is not equal to the values we have defined, namely\(\pi_1\), \(\pi_{a+1}\), or \(\pi_{n}\), of which none equals \(c\). Since \(c\neq n-a-b-1\), \(T^3(\pi)_1\neq \pi_n\). Since \(c\neq n-b-1\), \(T^3(\pi)_1 \neq \pi_1\). Since \(c\) is not \(n\) or \(n-1\), \(T^3(\pi)_1 \neq \pi_{a+1}\). So \(T^3(\pi)_1\) is not any of the elements we have defined in our construction, so we are free to define the element \(T^3(\pi)_1\) to be \(n-1\). We deduce that part of the necklace is "\(n\) \(a\) \(b\) \(c\) \(n-1\)",  and by Theorem \ref{thm:necklaces_with_n_or_n_minus_one_special_properties} the full necklace is \([a,b,c,n-1,c,b,a,n]\).

If \(b>n-a-1\), then \(T^2(\pi) = \seg{n-b}{<}{1}{n-b} \seg{a}{<}{n-b+1}{a+b-n} \seg{n}{<}{a+1}{n-a}\). Since \(c \neq n-b-1\), \(T^3(\pi)_1 \neq \pi_1\). Since \(c\) is not \(n\) or \(n-1\), \(T^3(\pi)_1 \neq \pi_{a+1}\). Since \(c\neq a\), \(T^3(\pi)_1 \neq \pi_n\). Again we deduce that we are free to choose that \(T^3(\pi)_1 = n-1\), so that part of the necklace is "\(n\) \(a\) \(b\) \(c\) \(n-1\)",  and by Theorem \ref{thm:necklaces_with_n_or_n_minus_one_special_properties} the full necklace is \([a,b,c,n-1,c,b,a,n]\).
\end{proof}

\begin{remark}
When \(n \leq 10\), all of the necklaces of size eight for \(S_n\) have one of the above forms. For \(n\) greater than \(10\) this is no longer the case. For example, \([2, 5, 6, 3, 8, 7, 4, 9]\) is valid in \(S_{11}\).
\end{remark}

\begin{lemma}\label{lem:eight_necklaces_same_of_n_n_minus_1}
The number of topdrop-valid necklaces of the form \([a,b,c,n,c,b,a,n]\) or \([a,b,c,n-1,c,b,a,n-1]\) with \(1 \leq a,b,c \leq n-2\), \(a+b+c = n-1\), and \(a\), \(b\), and \(c\) distinct is at least:
\begin{itemize}
\begin{item}
\(\frac{1}{2}n^2-4n+6\) if \(n\) is congruent to \(0\) or \(2\) modulo \(6\)
\end{item}
\begin{item}
\(\frac{1}{2}n^2-4n+\frac{19}{2}\) if \(n\) is congruent to \(1\) modulo \(6\)
\end{item}
\begin{item}
\(\frac{1}{2}n^2-4n+\frac{15}{2}\) if \(n\) is congruent to \(3\) or \(5\) modulo \(6\)
\end{item}
\begin{item}
\(\frac{1}{2}n^2-4n+8\) if \(n\) is congruent to \(4\) modulo \(6\)
\end{item}
\end{itemize}
\end{lemma}

\begin{proof}
It is sufficient to count the number of ordered triples \((a,b,c)\) with \(1 \leq a,b,c \leq n-2\), \(a+b+c = n-1\), and \(a\), \(b\), and \(c\) distinct. This quantity is twice the number of either of the two necklace forms, since a necklace which contains "\(a\) \(b\) \(c\)" also contains "\(c\) \(b\) \(a\)," but since we count both the necklaces with \(n\) and also the necklaces with \(n-1\), it is the correct count after all.

Fix \(a\). Since we want \(a+b+c=n-1\), and \(c\) must be at least \(1\), \(b\) must be less than or equal to \((n-1)-a-1\). This yields \(\sum_{a=1}^{n-2}{\sum_{b=1}^{n-a-2}{1}} = \frac{1}{2}n^2-\frac{5}{2}n+3\) possible choices of \(a\) and \(b\). Every choice of \(a\) and \(b\) induces a positive value of \(c\) so that \(a+b+c=n-1\). We need to determine how many of these pairs induce a value of \(c\) so that \(a\), \(b\), and \(c\) are all different values. We will count how many choices of \(a\), \(b\) and the induced value \(c\) have \(a=b\), \(a=c\), or \(b=c\) and subtract accordingly. These conditions are mutually exclusive unless \(a=b=c\), which we will take into account at the end.

First, we consider the possibility that \(a=b\). When \(a \leq \frac{(n-1)-1}{2}\), there is one choice of \(b\) where \(a=b\). When \(a > \frac{(n-1)-1}{2}\), there are no choices of \(b\) where \(a=b\), or else \(c\) would not be positive. This means there are \(\frac{n-2}{2}\) forbidden choices when \(n\) is even, and \(\frac{n-3}{2}\) forbidden choices when \(n\) is even.

By symmetry, on account of \(a=c\) there are \(\frac{n-2}{2}\) forbidden choices when \(n\) is even, and \(\frac{n-3}{2}\) forbidden choices when \(n\) is odd.

We consider the possibility that \(b=c\). This can only happen when \((n-1)-a\) is even, in which case there is a single value of \(b\) that induces a \(c\) so that \(b=c\). There are \(\frac{n-2}{2}\) such values of \(a\) which occur when \(n\) is even, and \(\frac{n-2}{2}\) values when \(n\) is odd.

The conditions \(a=b\), \(a=c\), and \(b=c\) are mutually exclusive unless \(a=b=c\). There is one triple \((a,b,c)\) with \(a=b=c\) and \(a+b+c=n-1\) if and only if \(n-1\) is a multiple of \(3\). When \(n-1\) is not a multiple of \(3\), we simply subtract the number of forbidden choices from each of the three conditions above, which after omitted simplifications yields the result. When \(n-1\) is a multiple of \(3\), we must subtract off two fewer choices so that we do not under count by subtracting the special triple three times. This is why we must separate the cases where \(n\) is congruent to \(1\) or \(4\) modulo \(6\) instead of simply having an even case and an odd case.
\end{proof}

\begin{lemma}\label{lem:eight_necklaces_different_of_n_n_minus_1}
The number of topdrop-valid necklaces of the form \([a,b,c,n-1,c,b,a,n]\) with \(1 \leq a,b,c \leq n-2\), \(a+b+c \neq n-1\), \(a+b\neq n-1\), \(b+c\neq n-1\), and \(a\), \(b\), and \(c\) distinct is at least:
\begin{itemize}
\begin{item}
\(n^3 - \frac{23}{2}n^2 + 42 n - 46\) if \(n\) is congruent to \(0\) or \(2\) modulo \(6\)
\end{item}
\begin{item}
\(n^3 - \frac{23}{2}n^2 + 44 n - \frac{115}{2}\) if \(n\) is congruent to \(1\) modulo \(6\)
\end{item}
\begin{item}
\(n^3 - \frac{23}{2}n^2 + 44 n - \frac{111}{2}\) if \(n\) is congruent to \(3\) or \(5\) modulo \(6\)
\end{item}
\begin{item}
\(n^3 - \frac{23}{2}n^2 + 42 n - 48\) is congruent to \(4\) modulo \(6\)
\end{item}
\end{itemize}
\end{lemma}

\begin{proof}
It is sufficient to count the number of ordered triples \((a,b,c)\) so that \(a+b+c \neq n-1\), \(a+b\neq n-1\), \(b+c \neq n-1\) and all three values are different. Since there are \((n-2)(n-3)(n-4) = n^3-9n^2+26n-24\) ways to choose \(3\) different values from \(n-2\) possible values, we compute our desired quantity by  subtracting the number of total triples with conditions \(a+b+c=n-1\), \(a+b = n-1\), and \(b+c = n-1\), and these conditions are mutually exclusive.

The number of triples with \(a+b+c=n-1\) is what was calculated in the previous lemma.

The number of triples with \(a+b=n-1\) is \((n-2)(n-4)\) if \(n\) is even, or \((n-3)(n-4)\) if \(n\) is odd. This is because for each of the \(n-2\) possible values of \(a\), there is one possible value of \(b\) for which \(a+b=n-1\), unless we would have \(a=b\), which happens for one value of \(a\) if \(n\) is odd. For each of the \(n-2\) or \(n-3\) choices of \(a\) and the corresponding \(b\), there are \(n-4\) choices for \(c\). We count the number of triples with \(b+c=n-1\) the same way.

So the number of necklaces of this form is \(n^3-9n^2+26n-24\) minus the quantity from the previous lemma, minus either \(2(n-2)(n-4)\) or \(2(n-3)(n-4)\) depending on whether \(n\) is even or odd. Omitting simplifications, this yields the result.
\end{proof}

\begin{proof}[Proof of part (6)]
By Theorem \ref{thm:orbit_counting}, the number of orbits of size eight is at least the number of necklaces counted in Lemma \ref{lem:eight_necklaces_same_of_n_n_minus_1} multiplied by \((n-4)!\), since there are four different values in each necklace and the fundamental period is one, plus the number of necklaces counted in Lemma \ref{lem:eight_necklaces_different_of_n_n_minus_1} multiplied by \((n-5)!\), since there are five different values in each necklace and the fundamental period is one. Simplifying these quantities yields the result.
\end{proof}

\section{Permutation Parity and the Topdrop-necklace}\label{sec:Permutation Parity and the Topdrop-necklace}
The topdrop map can be viewed as a piecewise map which permutes the elements of the input permutation according to its first element. For example, since \(213456\) starts with \(2\), applying the topdrop map to \(213456\) moves the first element to the last (sixth) position, the second element to the second-to-last (fifth) position, and every other element moves two positions towards the front, so that \(T(213456) = 345612\). This movement can be represented by the permutation \(\begin{pmatrix}
1 & 2 & 3 & 4 & 5 & 6 \\
6 & 5 & 1 & 2 & 3 & 4
\end{pmatrix}
\). Applying the topdrop map to \(345612\) moves the elements according to the permutation \(\begin{pmatrix}
1 & 2 & 3 & 4 & 5 & 6 \\
6 & 5 & 4 & 1 & 2 & 3
\end{pmatrix}\), and \(T(345612) = 612543\). This means that \(T^2\) moves the elements of \(213456\) according to the permutation \[\begin{pmatrix}
1 & 2 & 3 & 4 & 5 & 6 \\
6 & 5 & 4 & 1 & 2 & 3
\end{pmatrix}\begin{pmatrix}
1 & 2 & 3 & 4 & 5 & 6 \\
6 & 5 & 1 & 2 & 3 & 4
\end{pmatrix} = \begin{pmatrix}
1 & 2 & 3 & 4 & 5 & 6 \\
3 & 2 & 6 & 5 & 4 & 1
\end{pmatrix}.\] If we were to continue in this manner, we would eventually find that \[T^6(213456) = 213456,\] and the product of the six permutations which describe the movement of the elements at each iteration must equal the identity. We now generalize and formalize these observations.

We define notation to refer to the permutations which describe the movements we have discussed. The following definition gives the permutation which describes the movement when \(p\) is the first element of a permutation in \(S_n\).

\begin{definition}
Given natural numbers \(n\) and \(p\) with \(p\leq n\), define \[\sigma_{n,p} = \begin{pmatrix}
    1 & 2 & \ldots & p & p+1 & p+2 & \ldots & n \\
    n & n-1 & \ldots & n+1-p & 1 & 2 & \ldots & n-p 
\end{pmatrix}\]
\end{definition}

\begin{theorem}\label{thm:parity}
Let \([x_1, x_2, \ldots, x_k]\) be a topdrop-valid necklace in \(S_n\). If \(n\) is even, then evenly many \(x_i\)'s are congruent to \(1\) or \(2\) modulo \(4\). If \(n\) is odd, then evenly many \(x_i\)'s are congruent to \(2\) or \(3\) modulo \(4\).
\end{theorem}

\begin{lemma}
Let \(\pi\) be a permutation in \(S_n\) and let \(i,k\) be natural numbers with \(1 \leq i \leq n\) and \(k \geq 1\). Then \(\pi_i\) is the element in position \[[\sigma_{n,T^{k-1}(\pi)_1}\sigma_{n,T^{k-2}(\pi)_1}\cdots\sigma_{n,T^1(\pi)}\sigma_{n,\pi_1}](i)\] of \(T^{k}(\pi)_i\).
\end{lemma}

\begin{proof}
We prove by induction. When \(k = 1\), then it follows from the definition of the topdrop map that \(\pi_i\) is the element in position \(n+1-i\) of \(T^1(\pi)\) if \(1 \leq i \leq \pi_1\) or the element in position \(i-\pi_1\) of \(T^1(\pi)\) if \(\pi_1 + 1 \leq i \leq n\). By definition of \(\sigma_{n,\pi_1}\), we see that \(\pi_i\) is the element in position \(\sigma_{n,\pi_1}(i)\) of \(T^1(\pi)\).

Suppose that the statement is true for \(k=m\). By the definition of the topdrop map, \(T^{m}(\pi)_i\) is the element in position \(n+1-i\) of \(T^{m+1}(\pi)\) when \(1 \leq i \leq T^{m}(\pi)_1\) or in position \(i-T^{m}(\pi)_1\) of \(T^{m+1}(\pi)\) if \(T^{m}(\pi)_1 + 1 \leq i \leq n\). By definition of \(\sigma_{n,T^{m}(\pi)_1}\), we see that \(T^{m}(\pi)_i\) is the element in position \(\sigma_{n,T^{m}(\pi)_1}(i)\) of \(T^{m+1}(\pi)\). Combined with the inductive hypothesis, which is that \(\pi_i\) is the element in position \[[\sigma_{n,T^{m-1}(\pi)_1}\sigma_{n,T^{m-2}(\pi)_1}\cdots\sigma_{n,T^1(\pi)}\sigma_{n,\pi_1}](i)\] of \(T^{m}(\pi)_i\), we deduce that \(\pi_i\) is the element in position 
\begin{align*}
&\sigma_{n,T^{m}(\pi)_1}[\sigma_{n,T^{m-1}(\pi)_1}\sigma_{n,T^{m-2}(\pi)_1}\cdots\sigma_{n,T^1(\pi)}\sigma_{n,\pi_1}](i) \\
=& [\sigma_{n,T^{m}(\pi)_1}\sigma_{n,T^{m-1}(\pi)_1}\cdots\sigma_{n,T^1(\pi)}\sigma_{n,\pi_1}](i)
\end{align*} of \(T^{m+1}(\pi)_i\).
\end{proof}

\begin{lemma}\label{lem:idprod}
If \([x_0, x_1, \ldots, x_{m-1}]\) is a topdrop-valid necklace in \(S_n\), then \[\sigma_{n,x_{m-1}}\sigma_{n,x_{m-2}}\cdots\sigma_{n,x_0} = e,\] where \(e\) is the identity permutation.
\end{lemma}
\begin{proof}
If \([x_0, x_1, \ldots, x_{m-1}]\) is topdrop-valid in \(S_n\), then there is a \(\pi\) in \(S_n\) so that \(T^{k}(\pi)_1 = x_k\) when \(0 \leq k \leq m-1\). So \[[\sigma_{n,T^{k-1}(\pi)_1}\sigma_{n,T^{k-2}(\pi)_1}\cdots\sigma_{n,T^{0}(\pi)_1}]  = [\sigma_{n,x_{m-1}}\sigma_{n,x_{m-2}}\cdots\sigma_{n,x_0}].\] Since the size of the necklace is \(m\), \(T^m(\pi) = \pi\). So by the previous lemma, \(\pi_i\) is the element in position \[[\sigma_{n,x_{m-1}}\sigma_{n,x_{m-2}}\cdots\sigma_{n,x_0}](i)\] of \(\pi\), which means that \[[\sigma_{n,x_{m-1}}\sigma_{n,x_{m-2}}\cdots\sigma_{n,x_0}](i) = i\] for all \(i\), implying \[\sigma_{n,x_{m-1}}\sigma_{n,x_{m-2}}\cdots\sigma_{n,x_0} = e.\]
\end{proof}

\begin{remark}
The converse is not true. For example, \(\sigma_{2,1}\sigma_{2,1}  = e\), but \([1, 1]\) is not a topdrop-valid necklace in \(S_2\).
\end{remark}

\begin{proof}[Proof of theorem]
For natural numbers \(n\) and \(k\) with \(k \leq n\), we count the number of inversions in \(\sigma_{n,k}\), which are pairs \(i, j \leq n\) for which \(i < j\) but \(\sigma_{n,k}(i) > \sigma_{n,k}(j)\). We find that there are \(n-i\) inversions for each \(1 \leq i \leq k\), since \[\sigma_{n,k}(i) > \sigma_{n,k}(i) - (j-i) = \sigma_{n,k}(j)\] when \(i < j \leq n\) and \(j \leq k\), and \[\sigma_{n,k}(i) \geq n+1-k > n -k \geq \sigma_{n,k}(j)\] when \(i < j \leq n\) and \(j > k\), so that there is an inversion for every  pair \(i,j\) with \(1 \leq i \leq k\) and \(i < j \leq n\). We find that there are no inversions for \(i\) where \(k < i \leq n\), since \[\sigma_{n,k}(i) < \sigma_{n,k}(i) + (j - i) = \sigma_{n,k}(j)\] whenever \(i < j \leq n\).

So the total inversion count for \(\sigma_{n,k}\) is \[\sum_{i=1}^{k}(n-i) = kn - \frac{k(k+1)}{2}.\] It is easily verified that the inversion count is even when \(n\) is even and \(k\) is congruent to \(0\) or \(3\) modulo \(4\) or when\(n\) is odd and \(k\) is congruent to \(0\) or \(1\) modulo \(4\). The expression is odd when \(n\) is even and \(k\) is congruent to  \(1\) or \(2\) modulo \(4\) or when \(n\) is odd and \(k\) is congruent to \(2\) or \(3\) modulo \(4\).

Since the identity permutation is even, it is necessary that the number of odd permutations in the product of Lemma \ref{lem:idprod} must be even, from which the result follows.
\end{proof}

\begin{example}
The necklace \[[2,3,11,4,10,2,5,8,2,6,7]\text{ in } S_{14}\] is valid. Since \(14\) is even, we expect the total number of values congruent to \(1\) or \(2\) mod \(4\) to be even, i.e. \(1\)'s, \(2\)'s, \(5\)'s, \(6\)'s, \(9\)'s, \(10\)'s, \(13\)'s, and \(14\)'s. Summing over these numbers in order, we see that \[0+3+1+1+0+1+0+0=6\text{ is even.}\]
We can rule out the validity of the necklace \[[4,3,11,4,10,2,5,8,2,6,7]\text{ in } S_{14},\] because \[0+2+1+1+0+1+0+0=5\text{ is odd.}\]
\end{example}

\section{Future Directions}\label{sec:Future Directions}
The approach used to prove the orbit count results in Theorem \ref{thm:orbit_counts} could be extended to longer orbits. Within the existing data there are multiple necklace types that appear within reach. It should also be feasible to write a computer program which automates most of the casework, although this would not be a trivial task. At each iteration of the topdrop map, there are multiple cases depending on which segments or parts of segments of the permutation are reversed and moved. This produces a tree structure where each node is associated with a system of equalities/inequalities involving the values in the topdrop-necklace at that point. After enforcing constraints in each system so that \(T^m(\pi) = \pi\) for some chosen orbit size \(m\), the valid necklaces should correspond to the systems for which there is a solution.

Data for orbits and necklaces through \(S_14\) was generated using Python and SageMath. Theorem \ref{thm:orbits_with_necklace} and \ref{thm:orbit_counts} were used to vastly reduce the computational intensity: The code iterates through all permutations of a given length and stores the necklaces of the permutations, but it only stores a necklace if the necklace has not been seen before, as opposed to storing all orbits.

The graphs of the distributions of the number of orbits and necklaces of each orbit size look nice. Hopefully it is possible to say something about the distributions as a whole, although it is not yet known if they follow a known distribution. The first two graphs below plot the number of orbits\textemdash first the raw counts and then a log plot. The latter two plot the number of necklaces, where the second graph is also a log plot.

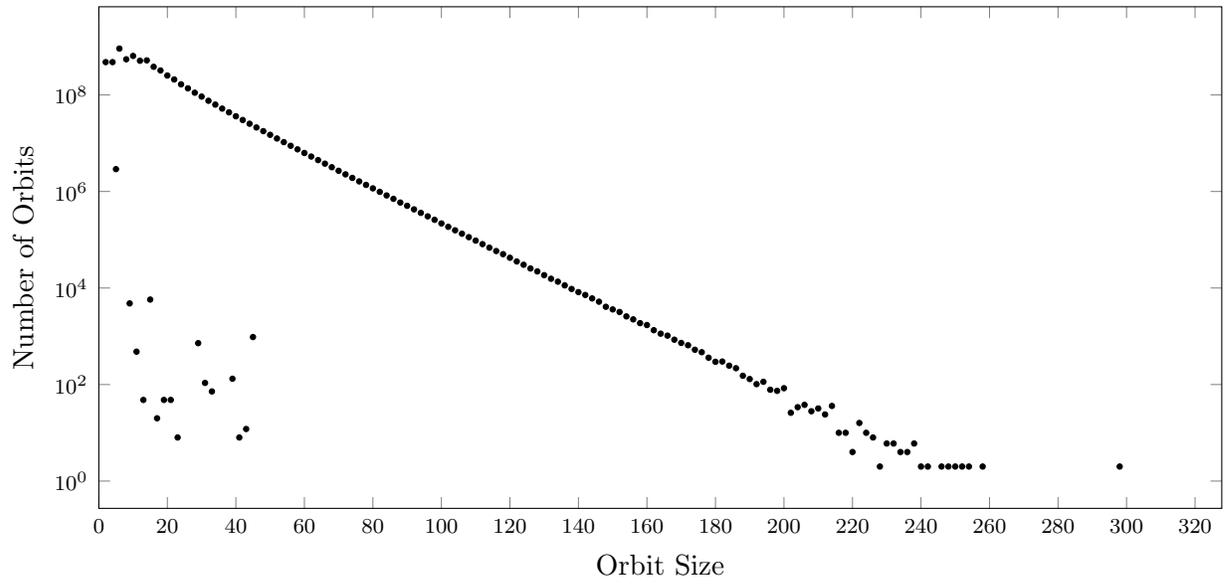
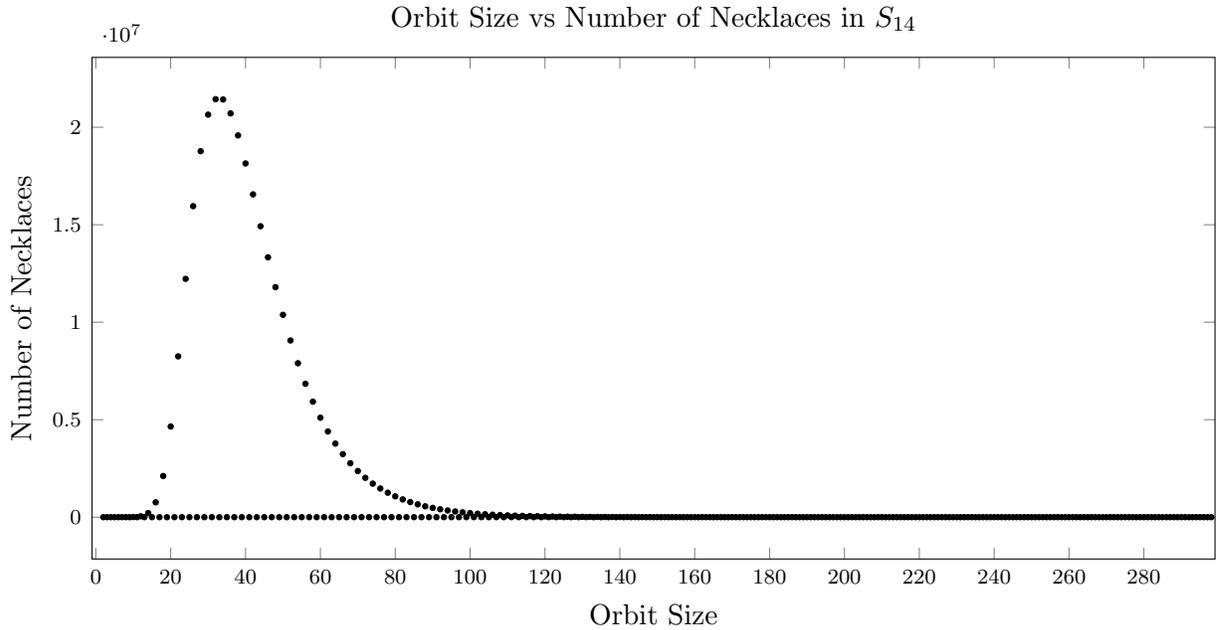
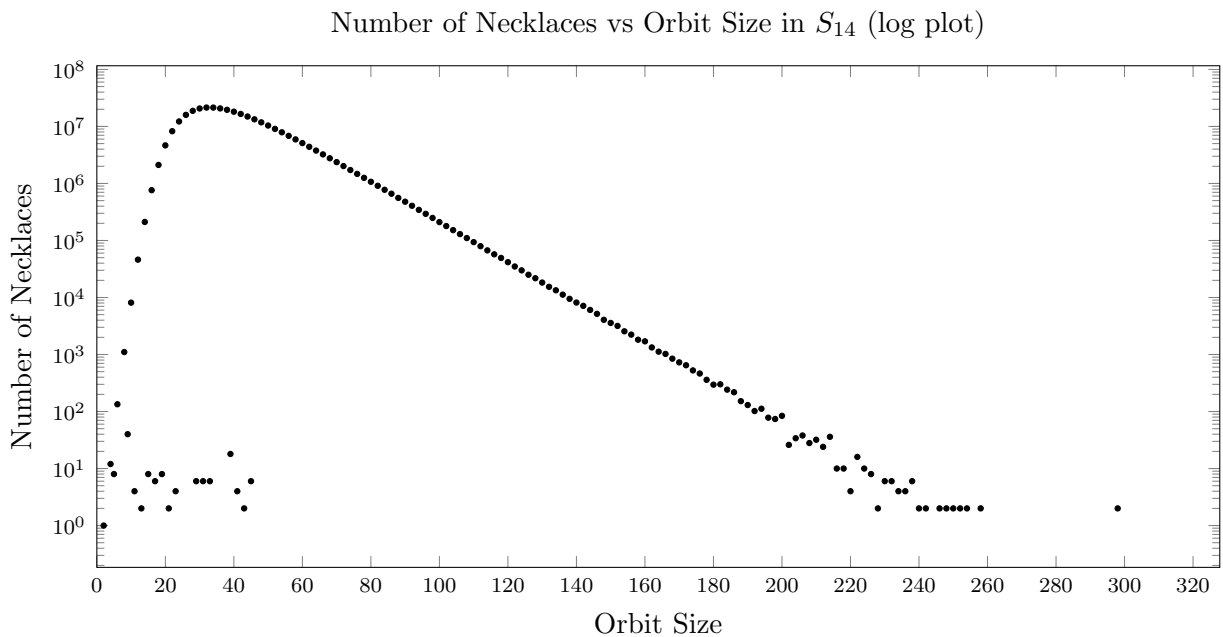
\begin{tikzpicture}
\begin{axis}[
    width=\textwidth,
    height=0.5\textwidth,
    xlabel={Orbit Size},
    ylabel={Number of Orbits},
    title={Orbit Size vs Number of Orbits in \(S_{14}\)},
    xmin=0,
    enlarge x limits={abs=1},
    tick label style={font=\tiny},
    label style={font=\small},
    title style={font=\small},
]
\addplot[
    only marks,
    mark=*,
    mark size=1pt,
]
table [col sep=comma, x=Orbit Size, y=Number of Orbits] {orbit_data.csv};
\end{axis}
\end{tikzpicture}

\begin{tikzpicture}
\begin{axis}[
    width=\textwidth,
    height=0.5\textwidth,
    xlabel={Orbit Size},
    ylabel={Number of Orbits},
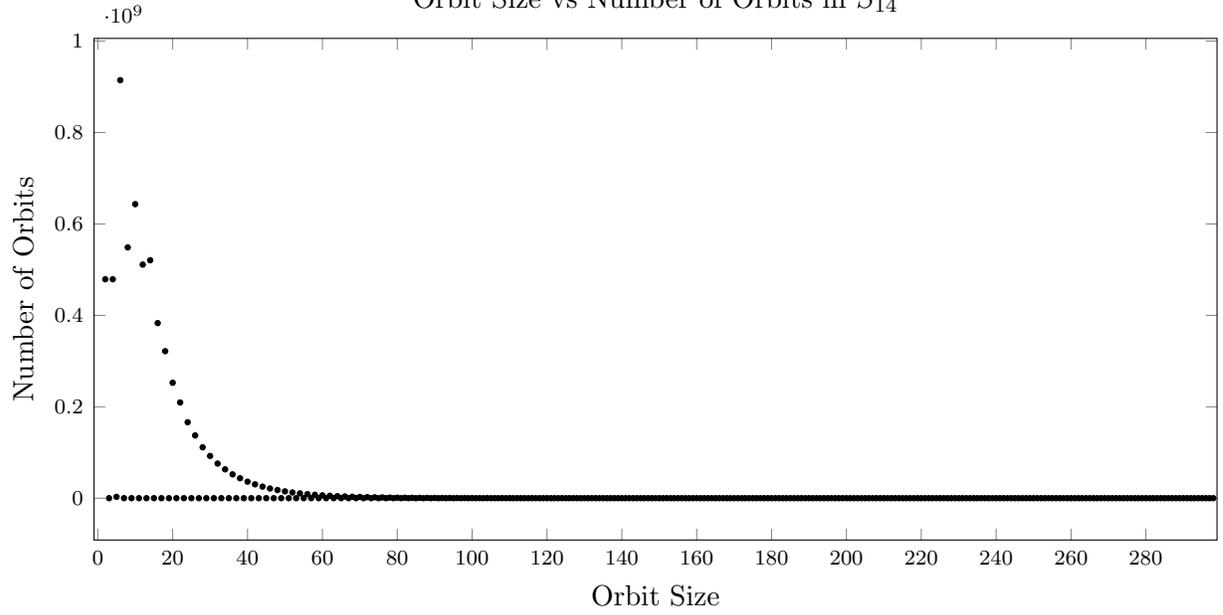
    title={Number of Orbits vs Orbit Size in \(S_{14}\) (log plot)},
    ymode=log,
    log basis y=10,
    tick label style={font=\tiny},
    label style={font=\small},
    title style={font=\small},
    xmin=0
]
\addplot[
    only marks,
    mark=*,
    mark size=1pt,
]
table [col sep=comma, x=Orbit Size, y=Number of Orbits] {orbit_data.csv};
\end{axis}
\end{tikzpicture}

\begin{tikzpicture}
\begin{axis}[
    width=\textwidth,
    height=0.5\textwidth,
    xlabel={Orbit Size},
    ylabel={Number of Necklaces},
    title={Orbit Size vs Number of Necklaces in \(S_{14}\)},
    xmin=0,
    enlarge x limits={abs=1},
    tick label style={font=\tiny},
    label style={font=\small},
    title style={font=\small},
]
\addplot[
    only marks,
    mark=*,
    mark size=1pt,
]
table [col sep=comma, x=Orbit Size, y=Number of Necklaces] {necklace_data.csv};
\end{axis}
\end{tikzpicture}

\begin{tikzpicture}
\begin{axis}[
    width=\textwidth,
    height=0.5\textwidth,
    xlabel={Orbit Size},
    ylabel={Number of Necklaces},
    title={Number of Necklaces vs Orbit Size in \(S_{14}\) (log plot)},
    ymode=log,
    log basis y=10,
    tick label style={font=\tiny},
    label style={font=\small},
    title style={font=\small},
    xmin=0
]
\addplot[
    only marks,
    mark=*,
    mark size=1pt,
]
table [col sep=comma, x=Orbit Size, y=Number of Necklaces] {necklace_data.csv};
\end{axis}
\end{tikzpicture}

As the topdrop map arises as a variant of the topswops problem, so there are more variants which need study. Gardner \cite{Gardner} lists a few additional variants, and in a variant of topswops involving two decks he makes an observation about loops of the top element of each deck that is reminiscent of the necklaces we have used here.

\section*{Acknowledgements}
The author thanks Jessica Striker, Daniel Biebighauser, and Victor Reiner for many helpful conversations.

\appendix
\section{History of Topswops}\label{sec:History of Topswops}
The topdrop map arises from a modification of the problem known as topswops. In 1976, David Berman and Murray S. Klamkin \cite{BermanKlamkinKnuth1976} described the following "reverse card shuffle:"

\begin{quote}
A deck of \(n\) cards is numbered 1 to \(n\) in random order. Perform the following operations on the deck. Whatever the number of the top card is, count down that
many in the deck and turn the whole block over on top of the remaining cards.
Then, whatever the number of the (new) top card, count down that many cards in
the deck and turn this whole block over on top of the remaining cards. Repeat-the
process.
\end{quote}

According to Donald Knuth in an editorial note \cite{BermanKlamkinKnuth1977}, this operation was in fact first proposed by John Conway with the name \textit{topswaps}  (more commonly spelled  \textit{topswops}).

Berman and Klamkin assert without showing a proof that eventually the card with value 1 will always reach the top of the deck, at which point further iterations do not change the order of the deck, and give without proof formulas for the number of starting permutations which require one, two, or three iterations for a 1 to arrive on top. They pose the problem of finding the maximum number of iterations before a 1 reaches the top, among other questions. They provide computational results for small deck sizes, and they reproduce a proof of a result from Knuth showing a Fibonacci upper bound on the maximum number of iterations.

An exact formula for the maximum number of iterations has not been found.  Linda Morales and Hal Sudborough \cite{MoralesSudborough} showed a quadratic lower bound in 2010. In 2021, Kimura et al. \cite{KimuraEtAl} computed maximum iterations through \(n=19\) by applying an algorithm developed by Knuth, and these remain the best exact calculations.

Topswops has been a subject of a few other research ventures. Yuichi Komano and Takaaki Mizuki \cite{KomanoMizuki} took research into topswops in a different direction in their 2024 paper demonstrating a zero-knowledge proof protocol to show that one knows a starting deck which requires a specified number of iterations until a 1 is on top. Also in 2024, Jonathan Parlett \cite{Parlett} introduced the term \textit{fixed point homing shuffles} to describe permutation maps which place the first element to its numerical position and which do not move higher positions, and topswops is an example.

Martin Gardner included topswops in his book ``Time Travel and other Mathematical Bewilderments'' published in 1988 \cite{Gardner}. Gardner recounts a proof from Herbert Wilf that topswops terminates with a 1 on top of the deck, and he also describes several variants, which are also attributed to Conway. One of these variants is topdrops, from which we arrive at the topdrop map.

\bibliographystyle{amsalpha}
\bibliography{main}

\end{document}